\documentclass[reqno,10pt,centertags,draft]{amsart}
\usepackage{amsmath,amsthm,amscd,amssymb,latexsym,upref,stmaryrd}
\newcommand{\bbC}{{\mathbb{C}}}

\newcommand{\bbN}{{\mathbb{N}}}
\newcommand{\bbR}{{\mathbb{R}}}

\newcommand{\cA}{{\mathcal A}}
\newcommand{\cB}{{\mathcal B}}

\newcommand{\cH}{{\mathcal H}}

\newcommand{\cU}{{\mathcal U}}

\newcommand{\cX}{{\mathcal X}}

\newcommand{\no}{\notag}
\newcommand{\lb}{\label}
\newcommand{\f}{\frac}

\newcommand{\ol}{\overline}

\newcommand{\wti}{\widetilde}
\newcommand{\Oh}{O}
\newcommand{\oh}{o}

\newcommand{\dom}{\text{\rm{dom}}}

\newcommand{\bi}{\bibitem}
\newcommand{\hatt}{\widehat}

\newcommand{\lbrac}{\left[\negthickspace}
\newcommand{\rbrac}{\negthickspace\right]}

\newcommand{\La}{\Lambda}
\newcommand{\al}{\alpha}
\newcommand{\be}{\beta}
\newcommand{\ga}{\gamma}

\newcommand{\de}{\delta}
\newcommand{\te}{\theta}
\newcommand{\Te}{\Theta}

\newcommand{\CR}{{\bbC\backslash\sigma(\Hte)}}

\newcommand{\bz}{\ensuremath{\bar z}}

\newcommand{\gate}{\ensuremath{\ga_{\te_0,\te_R}}}
\newcommand{\gates}{\ensuremath{\ga_{\te_0^{\prime},\te_R^{\prime}}}}
\newcommand{\gateq}{\ensuremath{\ga_{(\te_0+\frac{\pi}{2}) \,
\text{\rm mod} (2 \pi),(\te_R+\frac{\pi}{2}) \, \text{\rm mod} (2 \pi)}}}

\newcommand{\Lazz}{\ensuremath{\La_{0,0}}}
\newcommand{\Late}{\ensuremath{\La_{\te_0,\te_R}}}
\newcommand{\Lates}{\ensuremath{\La_{\te_0,\te_R}^{\te_0^{\prime},\te_R^{\prime}}}}
\newcommand{\Lateq}{\ensuremath{\La_{\te_0,\te_R}^{(\te_0+\frac{\pi}{2})
\, \text{\rm mod} (2 \pi),(\te_R+\frac{\pi}{2}) \, \text{\rm mod} (2 \pi)}}}
\newcommand{\Lazzqq}{\ensuremath{\La_{0,0}^{\frac{\pi}{2},\frac{\pi}{2}}}}
\newcommand{\Lade}{\ensuremath{\La_{\de_0,\de_R}}}
\newcommand{\Lades}{\ensuremath{\La_{\de_0,\de_R}^{\de_0^{\prime},\de_R^{\prime}}}}
\newcommand{\Cte}{\ensuremath{C_{\te_0,\te_R}}}
\newcommand{\Ctes}{\ensuremath{C_{\te_0^{\prime},\te_R^{\prime}}}}

\newcommand{\Ste}{\ensuremath{S_{\te_0,\te_R}}}
\newcommand{\Stes}{\ensuremath{S_{\te_0^{\prime},\te_R^{\prime}}}}

\newcommand{\Gte}{\ensuremath{G_{\te_0,\te_R}}}
\newcommand{\Hte}{\ensuremath{H_{\te_0,\te_R}}}
\newcommand{\tgate}{\ensuremath{\widehat \ga_{\te_0,\te_R}}}

\newcommand{\Rbzte}{\ensuremath{R_{\bar z,\te_0,\te_R}}}
\newcommand{\hg}{\ensuremath{\widehat g}}
\newcommand{\bg}{\ensuremath{\bar g}}

\renewcommand{\Re}{\text{\rm Re}}
\renewcommand{\Im}{\text{\rm Im}}

\renewcommand{\le}{\leqslant}

\newtheorem{theorem}{Theorem}[section]

\newtheorem{lemma}[theorem]{Lemma}
\newtheorem{corollary}[theorem]{Corollary}

\theoremstyle{definition}

\newtheorem{remark}[theorem]{Remark}

\newtheorem{example}[theorem]{Example}

\allowdisplaybreaks \numberwithin{equation}{section}

\begin{document}

\title[Boundary Data Maps for Schr\"odinger Operators on an Interval]
{Boundary Data Maps for Schr\"odinger Operators \\ on a Compact Interval}

\author[S.\ Clark, F.\ Gesztesy, and Marius Mitrea]{Stephen Clark, Fritz Gesztesy,
and Marius Mitrea}

\address{Department of Mathematics \& Statistics,
University of Missouri, Rolla, MO 65409, USA}
\email{sclark@umr.edu}
\urladdr{http://web.umr.edu/~sclark/index.html}

\address{Department of Mathematics,
University of Missouri, Columbia, MO 65211, USA}
\email{fritz@math.missouri.edu}
\urladdr{http://www.math.missouri.edu/personnel/faculty/gesztesyf.html}

\address{Department of Mathematics, University of
Missouri, Columbia, MO 65211, USA}
\email{mitream@missouri.edu}
\urladdr{http://www.math.missouri.edu/personnel/faculty/mitream.html}

\thanks{Based upon work partially supported by the US National Science
Foundation under Grant No.\ DMS-0653180.}

\thanks{{\it Math. Modelling Natural Phenomena} (to appear).}


%\thanks{\it .}

\date{\today}

%%%%%%%%%%%%%%%%%%%%%%%%%%%%%%%%%%%%%%%%
%%%%%%%%%%%%%%%%%%%%%%%%%%%%%%%%%%%%%%%%
\begin{abstract}
We provide a systematic study of boundary data maps, that is, $2 \times 2$ 
matrix-valued Dirichlet-to-Neumann and more generally, Robin-to-Robin maps,  
associated with one-dimensional Schr\"odinger operators on a compact 
interval $[0,R]$ with separated boundary conditions at $0$ and $R$. Most of 
our results are formulated in the non-self-adjoint context. 

Our principal results include explicit representations of these boundary 
data maps in terms of the resolvent of the underlying Schr\"odinger 
operator and the associated boundary trace maps, Krein-type resolvent 
formulas relating Schr\"odinger operators corresponding to different 
(separated) boundary conditions, and a derivation of the Herglotz 
property of boundary data maps (up to right multiplication by an appropriate 
diagonal matrix) in the special self-adjoint case. 
\end{abstract}
%%%%%%%%%%%%%%%%%%%%%%%%%%%%%%%%%%%%%%%%
%%%%%%%%%%%%%%%%%%%%%%%%%%%%%%%%%%%%%%%%

\maketitle

%%%%%%%%%%%%%%%%%%%%%%%%%%%%%%%%%%%%%%%%
%%%%%%%%%%%%%%%%%%%%%%%%%%%%%%%%%%%%%%%%
\section{Introduction}%\label{s1}
%%%%%%%%%%%%%%%%%%%%%%%%%%%%%%%%%%%%%%%%
%%%%%%%%%%%%%%%%%%%%%%%%%%%%%%%%%%%%%%%%

To briefly set the stage for this paper, let $R>0$, introduce  
the strip $S_{2 \pi}=\{z\in\bbC\,|\, 0\leq \Re(z) < 2 \pi\}$, and consider  
the boundary trace map 
\begin{equation} \lb{1.1}
\gate \colon \begin{cases}
C^1({[0,R]}) \rightarrow \bbC^2, \\
u \mapsto \begin{bmatrix} \cos(\te_0)u(0) + \sin(\te_0)u'(0)\\
\cos(\te_R)u(R) - \sin(\te_R)u'(R) \end{bmatrix}, \end{cases}
\quad  \te_0, \te_R\in S_{2 \pi},
\end{equation}
where ``prime'' denotes $d/dx$. In addition, assuming that 
\begin{equation}
V\in L^1((0,R); dx)    \lb{1.2}
\end{equation}
(we emphasize that $V$ is not assumed to be real-valued for most of this paper), 
one can introduce the family of one-dimensional Schr\"odinger operators
$\Hte$ in $L^2((0,R); dx)$ by
\begin{align}
& \Hte f= -f'' + Vf,  \quad \te_0, \te_R\in S_{2 \pi},     \no   \\
& f\in\dom(\Hte)=\big\{ g \in L^2((0,R); dx)\,\big|\, g, g'\in AC({[0,R]}); \, \gate (g)=0;  \label{1.3} \\
& \hspace*{6.75cm} (-g''+Vg)\in L^2((0,R); dx)\big\},   \no
\end{align}
were $AC([0,R])$ denotes the set of absolutely continuous functions on $[0,R]$.

Assuming that $z\in\CR$ (with $\sigma(T)$ denoting the spectrum of $T$) 
and $\te_0, \te_R \in S_{2 \pi}$, we recall that the boundary value problem given by
\begin{align} 
-&u'' + Vu=zu,\quad u, u'\in AC([0,R]),   \label{1.4}\\
&\gate(u)=\begin{bmatrix}c_0\\ c_R \end{bmatrix}\in\bbC^2,    \label{1.5}
\end{align}
has a unique solution denoted by 
$u(z,\cdot)=u(z,\cdot\,;(\te_0,c_0),(\te_R,c_R))$ for each $c_0, c_R\in\bbC$. To each boundary value problem \eqref{1.4}, \eqref{1.5}, we now associate a 
family of \emph{general boundary data maps}, 
$\Lates (z) : \bbC^2 \rightarrow \bbC^2$, for
$\te_0, \te_R, \te_0^{\prime},\te_R^{\prime}\in S_{2 \pi}$,  where
\begin{align}\label{1.6}
\begin{split}
\Lates (z) \begin{bmatrix}c_0\\ c_R \end{bmatrix} &=
\Lates (z) \big(\gate(u(z,\cdot\ ;(\te_0,c_0),(\te_R,c_R)))\big)   \\
&= \gates(u(z,\cdot\ ;(\te_0,c_0),(\te_R,c_R))).
\end{split}
\end{align}
With $u(z,\cdot)=u(z,\cdot\ ;(\te_0,c_0),(\te_R,c_R))$, then $\Lates (z) $ can be represented as a $2\times 2$ complex matrix, where
\begin{align}\label{1.7}
\Lates (z) \begin{bmatrix}c_0\\ c_R \end{bmatrix} &=
\Lates (z) \begin{bmatrix} \cos(\te_0)u(z,0) + \sin(\te_0)u'(z,0)\\[1mm]
\cos(\te_R)u(z,R) - \sin(\te_R)u'(z,R) \end{bmatrix}  \no \\
&=
\begin{bmatrix} \cos(\te_0^{\prime})u(z,0) + \sin(\te_0^{\prime})u'(z,0)\\[1mm]
\cos(\te_R^{\prime})u(z,R) - \sin(\te_R^{\prime})u'(z,R) \end{bmatrix}.
\end{align}

The map $\Lates (z)$ represents the principal object studied in this paper.

We prove in Section \ref{s2} that $\Lates (z)$ is well-defined for 
$z\in\bbC\backslash\sigma(\Hte)$, that is, it is invariant with respect to a 
change of basis of solutions of \eqref{1.4}, derive its basic properties 
(cf.\ Corollary \ref{c2.4}), and derive the explicit representation 
\eqref{2.20}, \eqref{2.20a} in terms of a distinguished basis of solutions 
\eqref{2.5}.

In Section \ref{s3}, we relate a special case of $\Lates (z)$, given by the 
generalized Dirichlet-to-Neumann maps $\Late (z) = \Lateq (z)$, to 
Weyl--Titchmarsh $m$-functions, derive its asymptotic behavior as 
$z$ tends to infinity, and most importantly, derive an explicit representation 
of the boundary data maps $\Lates (z)$ in terms of the resolvent of the 
underlying Schr\"odinger operator $\Hte$ and the associated boundary trace maps, 
\begin{align}
\begin{split} 
\Lates (z) S_{\theta_0' -\theta_0,\theta_R'-\theta_R}
=\gamma_{\te_0',\te_R'}
\big[\gamma_{\ol{\theta_0'},\ol{\theta_R'}} ((\Hte)^* - {\ol z} I)^{-1}\big]^*,& 
\lb{1.8}  \\[1mm] 
\te_0, \te_R, \te_0', \te_R' \in S_{2 \pi}, \; z\in\bbC\backslash\sigma(\Hte),& 
\end{split}
\end{align}
with $S_{\alpha,\beta}$ denoting the $2 \times 2$ diagonal matrix 
$S_{\alpha,\beta}= {\rm diag}\big(\sin(\alpha), \sin(\beta)\big)$. 

Theorem \ref{t4.1}, the principal result in Section \ref{s4}, then centers around 
the following linear fractional transformation relating the boundary data maps 
$\Lates (z)$ and $\Lades(z)$, 
\begin{align}
\begin{split}
\Lates(z) &= \big(S_{\de_0^\prime-\de_0,\de_R^\prime-\de_R}\big)^{-1}
\big[S_{\de_0^\prime-\te_0^\prime,\de_R^\prime-\te_R^\prime}+
S_{\te_0^\prime-\de_0,\te_R^\prime-\de_R}\Lades(z) \big]   \label{1.9} \\
& \quad \times
\big[S_{\de_0^\prime-\te_0,\de_R^\prime-\te_R}+
S_{\te_0-\de_0,\te_R-\de_R}\Lades(z) \big]^{-1}S_{\de_0^\prime-\de_0,\de_R^\prime-\de_R}, 
\end{split}
\end{align}
assuming $\te_0, \te_R, \te_0^{\prime},\te_R^{\prime}, \de_0,\de_R,\de_0^\prime,\de_R^\prime\in S_{2 \pi}$, $\de_0^\prime-\de_0\ne 0 \, \text{\rm mod} (\pi)$, 
$\de_R^\prime-\de_R \ne 0 \, \text{\rm mod} (\pi)$, and 
$z\in\bbC\backslash\big(\sigma(H_{\te_0,\te_R})\cup
\sigma(H_{\de_0,\de_R})\big)$. The linear fractional transformation \eqref{1.9} 
then is a major ingredient in our proof that 
$\Lates (\cdot) S_{\te_0'-\te_0,\te_R'-\te_R}$ is a $2 \times 2$ matrix-valued 
Herglotz function (i.e., analytic on $\bbC_+$, the open complex upper 
half-plane, with a nonnegative imaginary part) in the special case where 
$\Hte$ is self-adjoint. (In this case, one necessarily assumes that 
$\te_0, \te_R, \te_0', \te_R' \in [0,2 \pi)$ and that $V$ is real-valued.) In addition, 
we derive the Herglotz representation of  
$\Lates (\cdot) S_{\te_0'-\te_0,\te_R'-\te_R}$ in terms of a $2 \times 2$ 
matrix-valued measure on $\bbR$. 

The principal result proved in Section \ref{s6} then concerns Krein-type resolvent formulas explicitely relating the resolvents of $\Hte$ and 
$H_{\theta_0',\theta_R'}$. A typical result to be proved in Theorrem \ref{t6.3} is 
of the form 
\begin{align}
& (H_{\theta_0',\theta_R'} -z I)^{-1} = (\Hte -z I)^{-1}  
 - \big[\gamma_{\ol{\theta_0^{\prime}},\ol{\theta_R^{\prime}}}
((\Hte)^* - {\ol z} I)^{-1}\big]^* S_{\te'_0-\te'_0,\te'_R-\te_R}^{-1}    \no \\
 & \hspace*{5.7cm} \times \Big[\Lates (z)\Big]^{-1}
 \big[\gamma_{\theta_0^{\prime},\theta_R^{\prime}}(\Hte - z I)^{-1}\big],   \no \\
& \quad  \te_0, \te_R, \te_0', \te_R' \in S_{2 \pi}, \; 
\te_0 \neq \te_0', \; \te_R \neq \te_R', \; 
z\in\bbC\big\backslash\big(\sigma(\Hte)\cup \sigma(H_{\theta_0',\theta_R'})\big).      \lb{1.10}
\end{align}
Formula \eqref{1.10} demonstrates why $\Lates$ is the ideal object for 
Krein-type resolvent formulas.

Finally, in Section \ref{s7}, we describe some additional connections 
between $\Late (z)$ and the Green's function $G_{\te_0,\te_R} (z,x,x')$ of $\Hte$, 
and then point out some interesting differences compared to the standard 
$2 \times 2$ Weyl--Titchmarsh matrix associated with $\Hte$.

For classical as well as recent fundamental  literature on Weyl--Titchmarsh
operators (i.e., spectral parameter dependent
Dirichlet-to-Neumann maps, or more generally, Robin-to-Robin maps, resp.,
Poincar\'e--Steklov operators), relevant in the context
of boundary value spaces (boundary triples, etc.), we refer, for
instance, to \cite{ABMN05}, \cite{AB09}, \cite{AP04}, \cite{BL07}, \cite{BMN08},
\cite{BMN00}--\cite{BGP08}, \cite{DHMS00}-- \cite{DM95}, 
\cite{GKMT01}--\cite{GMZ07}, \cite{GT00}, 
\cite[Ch.\ 3]{GG91}, \cite{Gr08a}, \cite[Ch.\ 13]{Gr09}, 
\cite{KO77}--\cite{KS66}, \cite{LT77}, \cite{MM06}, \cite{Ma04},
\cite{Pa87}, \cite{Pa02}, \cite{Po04}, \cite{Po08}, \cite{PR09}, 
\cite{Ry07}--\cite{Ry10}, \cite{TS77} and the references cited 
therein. 

Finally, we briefly summarize some of the notation used in this paper: Let $\cH$ be a
separable complex Hilbert space, $(\cdot,\cdot)_{\cH}$ the scalar product in $\cH$
(linear in the second argument), and $I_{\cH}$ the identity operator in $\cH$.
Next, let $T$ be a linear operator mapping (a subspace of) a
Banach space into another, with $\dom(T)$ and $\ker(T)$ denoting the
domain and kernel (i.e., null space) of $T$. 
% The closure of a closable operator $S$ is denoted by $\ol S$. 
The spectrum 
%, essential spectrum, discrete spectrum, and resolvent set 
of a closed linear operator in $\cH$ will be denoted by $\sigma(\cdot)$. 
%, $\sigma_{\rm ess}(\cdot)$, $\sigma_{\rm d}(\cdot)$, and
% $\rho(\cdot)$, respectively. 
The Banach space of bounded linear operators on $\cH$ is
denoted by $\cB(\cH)$, the analogous notation $\cB(\cX_1,\cX_2)$,
will be used for bounded operators between two Banach spaces $\cX_1$ and
$\cX_2$. Moreover, $\cX_1\hookrightarrow \cX_2$ denotes the continuous 
embedding of $\cX_1$ into $\cX_2$.

%%%%%%%%%%%%%%%%%%%%%%%%%%%%%%%%%%%%%%%%
%%%%%%%%%%%%%%%%%%%%%%%%%%%%%%%%%%%%%%%%
\section{General Boundary Value Problems and Boundary Data Maps}   \label{s2}
%%%%%%%%%%%%%%%%%%%%%%%%%%%%%%%%%%%%%%%%
%%%%%%%%%%%%%%%%%%%%%%%%%%%%%%%%%%%%%%%%

This section is devoted to boundary data maps and their basic properties.

Taking $R>0$, and fixing $\te_0, \te_R\in S_{2 \pi}$, with $S_{2 \pi}$ the strip
\begin{equation}
S_{2 \pi}=\{z\in\bbC\,|\, 0\leq \Re(z) < 2 \pi\},
\end{equation}
we introduce the linear map $\gate$, the trace map associated with the boundary
$\{0,R\}$ of $(0,R)$ and the parameters $\te_0, \te_R$, by
\begin{equation}\label{2.2}
\gate \colon \begin{cases}
C^1({[0,R]}) \rightarrow \bbC^2, \\
u \mapsto \begin{bmatrix} \cos(\te_0)u(0) + \sin(\te_0)u'(0)\\
\cos(\te_R)u(R) - \sin(\te_R)u'(R) \end{bmatrix}, \end{cases}
\quad  \te_0, \te_R\in S_{2 \pi},
\end{equation}
where ``prime'' denotes $d/dx$. We note, in particular, that the Dirichlet trace
$\gamma_D$, and the Neumnann trace $\gamma_N$ (in connection with the outward pointing unit normal vector at $\partial (0,R) = \{0, R\}$), are given by
\begin{equation}
\gamma_D = \gamma_{0,0} = - \gamma_{\pi,\pi}, \quad
\gamma_N = \gamma_{3\pi/2,3\pi/2} = - \gamma_{\pi/2,\pi/2}.   \lb{2.2AA}
\end{equation}

Next, assuming
\begin{equation}
V\in L^1((0,R); dx),    \lb{2.2aa}
\end{equation}
we introduce the following family of densely defined closed linear operators
$\Hte$ in $L^2((0,R); dx)$,
\begin{align}
& \Hte f= -f'' + Vf,  \quad \te_0, \te_R\in S_{2 \pi},     \no   \\
& f\in\dom(\Hte)=\big\{ g \in L^2((0,R); dx)\,\big|\, g, g'\in AC({[0,R]}); \, \gate (g)=0;  \label{2.2a} \\
& \hspace*{6.75cm} (-g''+Vg)\in L^2((0,R); dx)\big\}.   \no
\end{align}
Here $AC([0,R])$ denotes the set of absolutely continuous functions on $[0,R]$.
We emphasize that $V$ is not assumed to be real-valued in the bulk of this paper.

One notices that
\begin{equation}
\gamma_{(\theta_0 + \pi) \, \text{\rm mod} (2\pi), (\theta_R + \pi) \, \text{\rm mod} (2\pi)}
= - \gate,  \quad  \te_0, \te_R\in S_{2 \pi},     \lb{2.2A}
\end{equation}
and, on the other hand,
\begin{equation}
H_{(\theta_0 + \pi) \, \text{\rm mod} (2\pi), (\theta_R + \pi) \, \text{\rm mod} (2\pi)} = \Hte,
\quad  \te_0, \te_R\in S_{2 \pi},     \lb{2.2B}
\end{equation}
hence it suffices to consider $\te_0, \te_R\in S_{\pi}=\{z\in\bbC\,|\, 0\leq \Re(z) < \pi\}$ rather than $\te_0, \te_R\in S_{2 \pi}$ in connection with $\Hte$, but for simplicity of notation we will keep using the strip $S_{2 \pi}$ throughout
this manuscript.

That $\Hte$ is indeed a closed operator follows, for instance, from \cite[Sect.\ XII.4]{DS88}, especially, by combining Lemma 5\,(c) and the first part of the proof of Lemma 26 and noting that $g(0), g'(0)$ (resp., $g(R), g'(R)$) are a complete set of boundary values for the minimal operator $H_{\rm min}$ associated with the differential expression
$-d^2/dx^2 + V(x)$ in $L^2((0,R); dx)$ at $x=0$ (resp., at $x=R$). Here
\begin{align}
& H_{\rm min} f= -f'' + Vf,   \no  \\
& f\in\dom(H_{\rm min})=\big\{ g \in L^2((0,R); dx)\,\big|\, g, g'\in AC({[0,R]});   \\
& \hspace*{1cm} g(0)=g'(0)=g(R)=g'(R)=0; \, (-g''+Vg)\in L^2((0,R); dx)\big\}.   \no
\end{align}
Morever, the adjoint of $\Hte$ is given by
\begin{align}
& (\Hte)^* f= -f'' + {\ol V} f,  \quad \te_0, \te_R\in S_{2 \pi},     \no   \\
& f\in\dom\big((\Hte)^*\big)=\big\{ g \in L^2((0,R); dx)\,\big|\, g, g'\in AC({[0,R]});
\, \gamma_{\ol{\theta_0},\ol{\theta_R}} (g)=0;  \no \\
& \hspace*{6.2cm}(-g''+{\ol V}g)\in L^2((0,R); dx)\big\}.
\end{align}

The fact that the spectrum of $\Hte$, $\sigma (\Hte)$, is discrete is well-known, but due to its importance in the context of this paper, we now briefly recall its proof following an argument in Marchenko \cite{Ma86}:

%%%%%%%%%%%%%
\begin{lemma} [See, \cite{Ma86}, Sect.\ 1.3]  \lb{l2.1}
Suppose $V\in L^1((0,R); dx)$, assume $\te_0, \te_R\in S_{2 \pi}$, and let $\Hte$ be defined as in \eqref{2.2a}. Then
$\sigma (\Hte)$ is an infinite discrete subset of $\bbC$ $($i.e., a set without any finite limit point in $\bbC$, but with a limit point at infinity$)$.
\end{lemma}
%%%%%%%%%%%%%
\begin{proof} Fix $z\in\bbC$.
Let $\theta(z,\cdot), \theta^{\prime}(z,\cdot), \phi(z,\cdot),
\phi^{\prime}(z,\cdot)\in AC([0,R])$, and such that $\theta (z,\cdot)$ and $\phi (z,\cdot)$ are solutions of
$-\psi'' +V \psi=z \psi$ uniquely determined by the initial values at $x=0$,
\begin{equation}
\theta (z,0)=\phi^{\prime}(z,0)=1, \quad \theta^{\prime}(z,0)=\phi (z,0)=0.  \lb{2.2b}
\end{equation}
Consequently, $\theta (z,\cdot)$  and $\phi (z,\cdot)$ are entire with respect to $z$. Introducing
\begin{equation}
\psi(z,\cdot) = A \theta (z,\cdot) + B \phi (z,\cdot), \quad A, B \in \bbC,
\end{equation}
it follows that
\begin{equation}
\gate(\psi)
= \begin{bmatrix} \cos(\te_0) \psi(z,0) + \sin(\te_0) \psi'(z,0) \\[1mm]
\cos(\te_R) \psi(z,R) - \sin(\te_R) \psi'(z,R) \end{bmatrix} = 0 \in \bbC^2.  \lb{2.2c}
\end{equation}
Employing the initial conditions \eqref{2.2b}, one concludes that equation \eqref{2.2c}
is equivalent to
\begin{align}
0 &= \begin{bmatrix}  \cos(\te_0) &  \sin(\te_0) \\[1mm]
\cos(\te_R) \theta (z,R) - \sin(\te_R) \theta' (z,R) &
\cos(\te_R) \phi(z,R) - \sin(\te_R) \phi'(z,R) \end{bmatrix}
\begin{bmatrix} A \\[1mm]  B \end{bmatrix}   \no \\
& = \cU(z, R, \te_0, \te_R) \begin{bmatrix} A \\ B \end{bmatrix}.    \lb{2.2ca}
\end{align}
Consequently, $z_0$ is an eigenvalue of $\Hte$ if and only if $z_0$ is a zero of the
determinant $\Delta$ defined as
\begin{equation}
\Delta(z,R,\te_0,\te_R) =\det\big(\cU(z, R, \te_0, \te_R)\big).   \lb{2.2det}
\end{equation}
Thus, $\Delta$ is an entire function with respect to $z$, and an explicit computation reveals that
\begin{align}
\begin{split}
\Delta(z,R,\te_0,\te_R)&=\cos(\te_0)\cos(\te_R)\phi (z,R)
- \cos(\te_0)\sin(\te_R) \phi^{\prime} (z,R)  \\
& \quad - \sin(\te_0)\cos(\te_R) \theta (z,R) + \sin(\te_0)\sin(\te_R) \theta^{\prime} (z,R).   \lb{2.2d}
\end{split}
\end{align}
The standard Volterra integral equations
\begin{align}
\theta (z,x) &= \cos(z^{1/2}x) + \int_0^x dx' \, \f{\sin(z^{1/2}(x-x'))}{z^{1/2}} V(x')
\theta (z,x'),  \\
\phi (z,x) &= \f{\sin(z^{1/2}x)}{z^{1/2}} + \int_0^x dx' \, \f{\sin(z^{1/2}(x-x'))}{z^{1/2}} V(x')
\phi (z,x'),   \\
& \hspace*{3.5cm}  z\in\bbC, \; \Im(z^{1/2}) \geq 0, \; x\in [0,R],   \no
\end{align}
then imply that
\begin{align}
\begin{split}
\theta (z,x) & \underset{|z|\to\infty}{=} \cos(z^{1/2}x)
+ \Oh\Big(|z|^{-1/2}e^{\Im(z^{1/2})x}\Big), \\
\theta^{\prime}(z,x) & \underset{|z|\to\infty}{=} - z^{1/2} \sin(z^{1/2}x)
+ \Oh\Big(e^{\Im(z^{1/2})x}\Big), \\
\phi (z,x) & \underset{|z|\to\infty}{=} \f{\sin(z^{1/2}x)}{z^{1/2}}
+ \Oh\Big(|z|^{-1}e^{\Im(z^{1/2})x}\Big),    \lb{2.2e}  \\
\phi^{\prime}(z,x) & \underset{|z|\to\infty}{=} \cos(z^{1/2}x)
+ \Oh\Big(|z|^{-1/2}e^{\Im(z^{1/2})x}\Big).
\end{split}
\end{align}
A comparison of \eqref{2.2d} as $|z|\to\infty$ and \eqref{2.2e} demonstrates that
$\Delta$ does not vanish identically. Thus the set of zeros of $\Delta$, and hence
the set of eigenvalues of $\Hte$, constitutes a discrete set. Again, the asymptotic behavior of $\Delta$ near infinity implies that $\Delta$ is an entire function of order
$1/2$ and hence possesses infinitely many zeros (cf., e.g., \cite[p.\ 252]{Ti85}).
\end{proof}
%%%%%%%%%%%%%

In addition (cf.\ also \eqref{3.32}, \eqref{3.33}), the resolvent of $\Hte$ is clearly a Hilbert--Schmidt operator in $L^2((0,R); dx)$. In fact, it is even a trace class operator since the eigenvalues $E_{\te_0,\te_R,n}$ of $\Hte$ in the case of the separated boundary conditions at hand  are of the form
$E_{\te_0,\te_R,n} = [(n\pi/R) + (a_n/n)]^2$ with $a_n\in \ell^\infty(\bbN)$
as $n\to \infty$, as shown in \cite[Lemma 1.3.3]{Ma86}.

Having described the operator $\Hte$ is some detail, still assuming \eqref{2.2aa},
we now briefly recall the corresponding closed, sectorial, and densely defined
sequilinear form, denoted by $Q_{\Hte}$, associated with $\Hte$ (cf.\
\cite[p.\ 312, 321, 327--328]{Ka80}):
\begin{align}
& Q_{\Hte}(f,g) = \int_0^R dx \big[\ol{f'(x)} g'(x) + V(x) \ol{f(x)} g(x)\big]   \no \\
& \hspace*{2.35cm} - \cot(\te_0) \ol{f(0)} g(0) -  \cot(\te_R) \ol{f(R)} g(R), \lb{2.2f} \\
& f, g \in \dom(Q_{\Hte}) = \dom\big(|\Hte|^{1/2}\big) = H^1((0,R))  \no \\
& \quad = \big\{h\in L^2((0,R); dx) \,|\,
h \in AC ([0,R]); \, h' \in L^2((0,R); dx)\big\},   \no \\
& \hspace*{6.5cm}  \te_0, \te_R \in S_{2 \pi}\backslash\{0,\pi\},   \no \\
& Q_{H_{0,\te_R}}(f,g) = \int_0^R dx \big[\ol{f'(x)} g'(x) + V(x) \ol{f(x)} g(x)\big]
-  \cot(\te_R) \ol{f(R)} g(R),  \lb{2.2g} \\
&  f, g \in \dom(Q_{H_{0,\te_R}}) = \dom\big(|H_{0,\te_R}|^{1/2}\big)   \no \\
& \quad = \big\{h\in L^2((0,R); dx) \,|\,
h \in AC ([0,R]); \, h(0) =0; \, h' \in L^2((0,R); dx)\big\},   \no \\
& \hspace*{8.5cm}  \te_R \in S_{2 \pi}\backslash\{0,\pi\},   \no \\
& Q_{H_{\te_0,0}}(f,g) = \int_0^R dx \big[\ol{f'(x)} g'(x) + V(x) \ol{f(x)} g(x)\big]
-  \cot(\te_0) \ol{f(0)} g(0),  \lb{2.2h} \\
&  f, g \in \dom(Q_{H_{\te_0,0}}) = \dom\big(|H_{\te_0,0}|^{1/2}\big)   \no \\
& \quad = \big\{h\in L^2((0,R); dx) \,|\,
h \in AC ([0,R]); \, h(R) =0; \, h' \in L^2((0,R); dx)\big\},   \no \\
& \hspace*{8.67cm}  \te_0 \in S_{2 \pi}\backslash\{0,\pi\},   \no \\
& Q_{H_{0,0}}(f,g) = \int_0^R dx \big[\ol{f'(x)} g'(x) + V(x) \ol{f(x)} g(x)\big],  \lb{2.2i} \\
& f, g \in \dom(Q_{H_{0,0}}) = \dom\big(|H_{0,0}|^{1/2}\big) = H^1_0((0,R))  \no \\
& \quad = \big\{h\in L^2((0,R); dx) \,|\,
h \in AC([0,R]); \, h(0)=0, \, h(R) =0;   \no \\
& \hspace*{6.65cm}   h' \in L^2((0,R); dx)\big\}.   \no
\end{align}

Equations \eqref{2.2f}--\eqref{2.2i} follow from the fact that for any $\varepsilon>0$,
there exists $\eta(\varepsilon)>0$ such that for all $h \in H^1((0,R))$,
\begin{align}
& |h(x_0)| \leq \varepsilon \|h' \|_{L^2((0,R); dx)} +
\eta(\varepsilon) \|h \|_{L^2((0,R); dx)}, \quad x_0 \in [0,R],    \\
& \||V|^{1/2} h\|_{L^2((0,R); dx)} \leq \varepsilon \|h' \|_{L^2((0,R); dx)} +
\eta(\varepsilon) \|h \|_{L^2((0,R); dx)}
\end{align}
(cf.\ \cite[p.\ 193, 345--346]{Ka80}).

Next, we recall the following elementary, yet fundamental, fact:

%%%%%%%%%%%%%
\begin{lemma} \lb{l2.2}
Suppose that $V\in L^1((0,R); dx)$, fix $\te_0, \te_R\in S_{2 \pi}$, and assume
that $z\in\CR$. Then the boundary value problem given by
\begin{align}%\label{}
-&u'' + Vu=zu,\quad u, u'\in AC([0,R]),   \label{2.3}\\
&\gate(u)=\begin{bmatrix}c_0\\ c_R \end{bmatrix}\in\bbC^2,    \label{2.4}
\end{align}
has a unique solution $u(z,\cdot)=u(z,\cdot\,;(\te_0,c_0),(\te_R,c_R))$ for each $c_0, c_R\in\bbC$.
\end{lemma}
%%%%%%%%%%%%%
\begin{proof}
This is well-known, but for 
the sake of completeness, we briefly recall the argument: Let $\psi_j(z,\cdot)$, $j=1,2$, be a basis for the solutions of \eqref{2.3} and let
$\psi(z,\cdot)=A\psi_1(z,\cdot)+B\psi_2(z,\cdot)$, $A,B\in\bbC$, be the general solution of
\eqref{2.3}. Then
\begin{align}
\begin{split}
\gate(\psi(z,\cdot)) &= \begin{bmatrix}\gate(\psi_1(z,\cdot)) & \gate(\psi_2(z,\cdot))
\end{bmatrix} \begin{bmatrix} A \\ B \end{bmatrix}   \\
& = \begin{bmatrix} M_{1,1}(z) & M_{1,2}(z) \\ M_{2,1}(z) & M_{2,2}(z)
\end{bmatrix} \begin{bmatrix} A\\ B \end{bmatrix},     \label{2.4a}
\end{split}
\end{align}
where
\begin{align}
\begin{split}
M_{1,1}(z) &= \cos(\te_0)\psi_1(z,0) + \sin(\te_0)\psi_1'(z,0),   \\
M_{1,2}(z) &= \cos(\te_0)\psi_2(z,0) + \sin(\te_0)\psi_2'(z,0),   \\
M_{2,1}(z) &= \cos(\te_R)\psi_1(z,R) - \sin(\te_R)\psi_1'(z,R),  \\
M_{2,2}(z) &= \cos(\te_R)\psi_2(z,R)  - \sin(\te_R)\psi_2'(z,R).
\end{split}
\end{align}
Thus, prescribing $c_0, c_R \in\bbC$, the equation
\begin{equation}
\gate(\psi(z,\cdot))=\begin{bmatrix}c_0\\ c_R \end{bmatrix}
\end{equation}
is uniquely solvable in terms of some $A,B\in\bbC$ if and only if
\begin{align}
\det\big(\begin{bmatrix}\gate(\psi_1(z,\cdot)) & \gate(\psi_2(z,\cdot))\end{bmatrix}\big)
=\det\left(\begin{bmatrix} M_{1,1}(z) & M_{1,2}(z) \\ M_{2,1}(z) & M_{2,2}(z)
\end{bmatrix} \right)  \ne 0.   \label{2.10}
\end{align}

On the other hand, this determinant equals zero for some $z_0\in\bbC$, if and only if  there is a nonzero vector
$\begin{bmatrix} A_0 & B_0 \end{bmatrix}^\top \in\bbC^2$ such that
\begin{equation}
\begin{bmatrix}\gate(\psi_1(z_0,\cdot)) & \gate(\psi_2(z_0,\cdot))\end{bmatrix}
\begin{bmatrix} A_0 \\ B_0 \end{bmatrix} =0
\end{equation}
which is equivalent to the existence of a nonzero solution
$\psi_0(z_0,\cdot)= A_0 \psi_1(z_0,\cdot) + B_0 \psi_2(z_0,\cdot)$ of the corresponding boundary value problem given by \eqref{2.3} and \eqref{2.4} with $z=z_0$ and homogeneous boundary conditions (i.e., with $c_0=c_R=0$). Equivalently, $\psi_0(z_0,\cdot)$ satisfies
\begin{equation}
\Hte \psi_0(z_0,\cdot) = z_0 \psi_0(z_0,\cdot), \quad \psi_0(z_0,\cdot) \in \dom(\Hte),
\end{equation}
which in turn is equivalent to $z_0\in\sigma(\Hte)$.
\end{proof}
%%%%%%%%%%%%%

Assuming $z\in\bbC\backslash\sigma(\Hte)$, a basis for the solutions of
\eqref{2.3} is given by
\begin{equation}\label{2.5}
\begin{split}
u_{-,\te_0}(z,\cdot)&=u(z,\cdot \, ;(\te_0,0),(0,1)),    \\
u_{+,\te_R}(z,\cdot)&=u(z,\cdot \, ;(0,1),(\te_R,0)).
\end{split}
\end{equation}
Explicitly, one then has
\begin{align}
u_{-,\te_0}(z,R)&= 1, \quad
\cos(\te_0)u_{-,\te_0}(z,0) + \sin(\te_0)u_{-,\te_0}'(z,0) =0,    \label{2.5a}\\
u_{+,\te_R}(z,0)&= 1, \quad
\cos(\te_R)u_{+,\te_R}(z,R) - \sin(\te_R)u_{+,\te_R}'(z,R) =0.     \label{2.5b}
\end{align}
Recalling the Wronskian of two functions $f$ and $g$,
\begin{equation}
W(f,g)(x) = f(x)g'(x) - f'(x)g(x), \quad f, g \in C^1([0,R]), 
\end{equation}
one then computes
\begin{align}
& W(u_{+,\te_R}(z,\cdot), u_{-,\te_0}(z,\cdot))   \no \\
& \quad = u_{+,\te_R}(z,x) u'_{-,\te_0}(z,x)- u'_{+,\te_R}(z,x) u_{-,\te_0}(z,x) \neq 0,
\quad x \in [0,R],  \no \\
& \quad = u'_{-,\te_0}(z,0)- u'_{+,\te_R}(z,0), u_{-,\te_0}(z,0)     \lb{2.5c} \\
& \quad = u_{+,\te_R}(z,R) u'_{-,\te_0}(z,R)- u'_{+,\te_R}(z,R).   \lb{2.5d}
\end{align}

To each boundary value problem \eqref{2.3}, \eqref{2.4}, we now associate a family of \emph{general boundary data maps}, $\Lates (z) : \bbC^2 \rightarrow \bbC^2$, for
$\te_0, \te_R, \te_0^{\prime},\te_R^{\prime}\in S_{2 \pi}$,  where
\begin{align}\label{2.6}
\begin{split}
\Lates (z) \begin{bmatrix}c_0\\ c_R \end{bmatrix} &=
\Lates (z) \big(\gate(u(z,\cdot\ ;(\te_0,c_0),(\te_R,c_R)))\big)   \\
&= \gates(u(z,\cdot\ ;(\te_0,c_0),(\te_R,c_R))).
\end{split}
\end{align}
With $u(z,\cdot)=u(z,\cdot\ ;(\te_0,c_0),(\te_R,c_R))$, then $\Lates (z) $ can be represented as a $2\times 2$ complex matrix, where
\begin{align}\label{2.7}
\Lates (z) \begin{bmatrix}c_0\\ c_R \end{bmatrix} &=
\Lates (z) \begin{bmatrix} \cos(\te_0)u(z,0) + \sin(\te_0)u'(z,0)\\[1mm]
\cos(\te_R)u(z,R) - \sin(\te_R)u'(z,R) \end{bmatrix}  \no \\
&=
\begin{bmatrix} \cos(\te_0^{\prime})u(z,0) + \sin(\te_0^{\prime})u'(z,0)\\[1mm]
\cos(\te_R^{\prime})u(z,R) - \sin(\te_R^{\prime})u'(z,R) \end{bmatrix}.
\end{align}

The following result shows that $\Lates$ is well-defined for 
$z\in\bbC\backslash\sigma(\Hte)$, that is, it is invariant with respect to a change of basis of solutions of \eqref{2.3}.

%%%%%%%%%%%%%%%%%%%%%%%%%%%%%%%%%%%%
\begin{theorem}\label{t2.3}
Let $\te_0, \te_R, \te_0^{\prime},\te_R^{\prime}\in S_{2 \pi}$ and
$z\in\bbC\backslash\sigma(\Hte)$. In addition,
denote by $\psi_j(z,\cdot)$, $j=1,2$, a basis for the solutions of \eqref{2.3}. Then,
\begin{align}\label{2.9}
& \Lates (z)  \\
& \;\; =\begin{bmatrix}\cos(\te_0^{\prime})\psi_1(z,0) + \sin(\te_0^{\prime})\psi_1'(z,0)&\cos(\te_0^{\prime})\psi_2(z,0) + \sin(\te_0^{\prime})\psi_2'(z,0)\\[1mm]
\cos(\te_R^{\prime})\psi_1(z,R) - \sin(\te_R^{\prime})\psi_1'(z,R)&\cos(\te_R^{\prime})\psi_2(z,R) - \sin(\te_R^{\prime})\psi_2'(z,R)\end{bmatrix}   \no \\
& \;\;\;\; \times
\begin{bmatrix}\cos(\te_0)\psi_1(z,0) + \sin(\te_0)\psi_1'(z,0)&\cos(\te_0)\psi_2(z,0) + \sin(\te_0)\psi_2'(z,0)\\[1mm]
\cos(\te_R)\psi_1(z,R) - \sin(\te_R)\psi_1'(z,R)&\cos(\te_R)\psi_2(z,R) - \sin(\te_R)\psi_2'(z,R)\end{bmatrix}^{-1}.  \no
\end{align}
Moreover, $\Lates (z) $ is invariant with respect to a change of basis for the solutions of \eqref{2.3}.
\end{theorem}
%%%%%%%%%%%%%%%%%%%%%%%%%%%%%%%%%%%%
\begin{proof}
Letting $\psi(z,\cdot)=A\psi_1(z,\cdot)+B\psi_2(z,\cdot)$, $A,B\in\bbC$, be an arbitrary solution of \eqref{2.3}, one observes, by \eqref{2.4a} and \eqref{2.7}, that the equation
$\Lates(\gate(\psi))=\gates(\psi)$ becomes
\begin{align}\label{2.13}
& \Lates(z)   \\
& \; \times
\begin{bmatrix}\cos(\te_0)\psi_1(z,0) + \sin(\te_0)\psi_1'(z,0)&\cos(\te_0)\psi_2(z,0) + \sin(\te_0)\psi_2'(z,0)\\[1mm]
\cos(\te_R)\psi_1(z,R) - \sin(\te_R)\psi_1'(z,R)&\cos(\te_R)\psi_2(z,R) - \sin(\te_R)\psi_2'(z,R)\end{bmatrix}\begin{bmatrix}A\\B\end{bmatrix}  \no \\
& \;\;
=\begin{bmatrix}\cos(\te_0^{\prime})\psi_1(z,0) + \sin(\te_0^{\prime})\psi_1'(z,0)&\cos(\te_0^{\prime})\psi_2(z,0) + \sin(\te_0^{\prime})\psi_2'(z,0)\\[1mm]
\cos(\te_R^{\prime})\psi_1(z,R) - \sin(\te_R^{\prime})\psi_1'(z,R)&\cos(\te_R^{\prime})\psi_2(z,R) - \sin(\te_R^{\prime})\psi_2'(z,R)\end{bmatrix}\begin{bmatrix}A\\B\end{bmatrix}  \no
\end{align}
for every $\begin{bmatrix}A & B\end{bmatrix}^\top \in\bbC^2$. Equation \eqref{2.9} then follows by the invertibility of $\begin{bmatrix}\gate(\psi_1) & \gate(\psi_2)\end{bmatrix}$ noted in \eqref{2.10}.

Let $\phi_j(z,\cdot)$, $j=1,2$, denote a second basis for the solutions of \eqref{2.3}. Then, there is a nonsingular matrix $K\in\bbC^{2\times 2}$ such that
$\begin{bmatrix}\psi_1 & \psi_2\end{bmatrix}=\begin{bmatrix}\phi_1 & \phi_2
\end{bmatrix}K$. Next, we introduce for each pair, $\te_0, \te_R\in S_{2 \pi}$,  the following matrices
 \begin{equation}\label{2.14}
\Cte=\begin{bmatrix}\cos(\te_0)&0\\0&\cos(\te_R) \end{bmatrix},\quad 
\Ste=\begin{bmatrix} \sin(\te_0)&0\\0&\sin(\te_R)\end{bmatrix}.
 \end{equation}
Introducing $\te_j$ to denote $\te_0$, and $\te_k$ to denote $\te_R$, respectively; or, using $\te_j$ to denote $\te_0^{\prime}$ and $\te_k$ to denote $\te_R^{\prime}$, one computes,
\begin{align}
& \begin{bmatrix}\ga_{\te_j,\te_k}(\psi_1(z,\cdot)) & \ga_{\te_j,\te_k}(\psi_2(z,\cdot))\end{bmatrix}   \no \\
& \quad = C_{\te_j,\te_k}\begin{bmatrix} \psi_1(z,0)&\psi_2(z,0)\\ \psi_1(z,R)&\psi_2(z,R)\end{bmatrix} +
S_{\te_j,\te_k}\begin{bmatrix} \psi'_1(z,0)&\psi'_2(z,0)\\ -\psi'_1(z,R)&-\psi'_2(z,R)\end{bmatrix}  \no  \\
& \quad = \bigg(C_{\te_j,\te_k}\begin{bmatrix} \phi_1(z,0)&\phi_2(z,0)\\ \phi_1(z,R)&\phi_2(z,R)\end{bmatrix} +
S_{\te_j,\te_k}\begin{bmatrix} \phi'_1(z,0)&\phi'_2(z,0)\\ -\phi'_1(z,R)&-\phi'_2(z,R)\end{bmatrix}\bigg)K   \no \\
& \quad = \begin{bmatrix}\ga_{\te_j,\te_k}(\phi_1(z,\cdot)) & \ga_{\te_j,\te_k}(\phi_2(z,\cdot))
\end{bmatrix} K.    \label{2.16}
\end{align}
As defined in \eqref{2.9}, $\Lates=\begin{bmatrix}\ga_{\te_0^{\prime},\te_R^{\prime}}(\psi_1) & \ga_{\te_0^{\prime},\te_R^{\prime}}(\psi_2)\end{bmatrix}\begin{bmatrix}
\ga_{\te_0,\te_R}(\psi_1) & \ga_{\te_0,\te_R}(\psi_2)\end{bmatrix}^{-1}$, and by
\eqref{2.16},
\begin{align}\label{2.32}
\begin{split}
\Lates & = \begin{bmatrix}\ga_{\te_0^{\prime},\te_R^{\prime}}(\psi_1) &
\ga_{\te_0^{\prime},\te_R^{\prime}}(\psi_2)\end{bmatrix}\begin{bmatrix}\ga_{\te_0,\te_R}(\psi_1) &\ga_{\te_0,\te_R}(\psi_2)\end{bmatrix}^{-1}   \\
& =\begin{bmatrix}\ga_{\te_0^{\prime},\te_R^{\prime}}(\phi_1) &
\ga_{\te_0^{\prime},\te_R^{\prime}}(\phi_2)\end{bmatrix}\begin{bmatrix}[\ga_{\te_0,\te_R}(\phi_1) & \ga_{\te_0,\te_R}(\phi_2)\end{bmatrix}^{-1},
\end{split}
\end{align}
completing the proof.
\end{proof}
%%%%%%%%%%%%%%%%%%%%%%%%%%%%%%%%%%

Theorem \ref{t2.3} then readily implies the following result:

%%%%%%%%%%%%%%%%%%%%%%%%%%%%%%%%%%
\begin{corollary} \lb{c2.4}
Let $\te_0,\te_R, \te_0^{\prime},\te_R^{\prime}, \te_0^{\prime\prime}, \te_R^{\prime\prime}\in S_{2 \pi}$. Then, with $I_2$ denoting the identity matrix in $\bbC^2$,
\begin{align}%\label{}
& \La_{\te_0,\te_R}^{\te_0,\te_R} (z) = I_2,   \quad
z\in\bbC \backslash \sigma(\Hte),      \\
& \La_{\te_0^{\prime},\te_R^{\prime}}^{\te_0^{\prime\prime},\te_R^{\prime\prime}} (z)
\Lates (z) =\La_{\te_0,\te_R}^{\te_0^{\prime\prime},\te_R^{\prime\prime}} (z) , \quad
z\in\bbC\backslash\big(\sigma(\Hte)\cup\sigma(H_{\theta'_0,\theta'_R})\big),    \\
& \La_{\te_0^{\prime},\te_R^{\prime}}^{\te_0,\te_R} (z) = \Big[\Lates (z)\Big]^{-1}, \quad
z\in\bbC\backslash\big(\sigma(\Hte)\cup\sigma(H_{\theta'_0,\theta'_R})\big).  \lb{2.48}
\end{align}
\end{corollary}
%%%%%%%%%%%%%%%%%%%%%%%%%%%%%%%%%%

%%%%%%%%%%%%%%%%%%%%%%%%%%%%%%%%%%
\begin{remark} \lb{r2.5}
By Theorem \ref{t2.3},  $\Lates$ is invariant with respect to a change of basis for the solutions of \eqref{2.3}. However, the representation of $\Lates$ with respect to a specific basis can be simplified considerably with an appropriate choice of basis. For example, by choosing the basis given in \eqref{2.5}, and by letting
$\psi_1(z,\cdot)=u_{+,\te_R}(z,\cdot)=u(z,\cdot \, ;(0,1),(\te_R,0))$, and
$\psi_2(z,\cdot)=u_{-,\te_0}(z,\cdot)=u(z,\cdot \, ;(\te_0,0),(0,1))$, \eqref{2.32} 
implies that, entrywise, 
\begin{align}%\label{}
\begin{split}
& \begin{bmatrix}\gate(u_{+,\te_R}(z,\cdot))
&  \gate(u_{-,\te_0}(z,\cdot))\end{bmatrix}_{1,1}   \\
& \quad = \cos(\te_0) + \sin(\te_0)u_{+,\te_R}'(z,0),   \\
& \begin{bmatrix}\gate(u_{+,\te_R}(z,\cdot))
&  \gate(u_{-,\te_0}(z,\cdot))\end{bmatrix}_{1,2}   \\
& \quad = \cos(\te_0)u_{-,\te_0}(z,0) + \sin(\te_0)u_{-,\te_0}'(z,0),    \\
& \begin{bmatrix}\gate(u_{+,\te_R}(z,\cdot))
&  \gate(u_{-,\te_0}(z,\cdot))\end{bmatrix}_{2,1}    \\
& \quad = \cos(\te_R)u_{+,\te_R}(z,R) - \sin(\te_R)u_{+,\te_R}'(z,R),    \\
& \begin{bmatrix}\gate(u_{+,\te_R}(z,\cdot))
&  \gate(u_{-,\te_0}(z,\cdot))\end{bmatrix}_{2,1}    \\
& \quad = \cos(\te_R) - \sin(\te_R)u_{-,\te_0}'(z,R).
\end{split}
\end{align}
Hence, for this basis,
\begin{align}
& \Lates (z) = \Big[\Big(\Lates (z)\Big)_{j,k}\Big]_{1 \leq j,k \leq 2},
\quad z\in\bbC\backslash\sigma(\Hte),    \label{2.20} \\[1mm]
\begin{split}
& \Big(\Lates (z)\Big)_{1,1} = \frac{\cos(\te_0^{\prime})
+ \sin(\te_0^{\prime})u_{+,\te_R}'(z,0)}{\cos(\te_0) + \sin(\te_0)u_{+,\te_R}'(z,0)},  \\
& \Big(\Lates (z)\Big)_{1,2} = \frac{\cos(\te_0^{\prime})u_{-,\te_0}(z,0) + \sin(\te_0^{\prime})u_{-,\te_0}'(z,0)}{\cos(\te_R) - \sin(\te_R)u_{-,\te_0}'(z,R)},    \\
& \Big(\Lates (z)\Big)_{2,1} = \frac{\cos(\te_R^{\prime})u_{+,\te_R}(z,R)
- \sin(\te_R^{\prime})u_{+,\te_R}'(z,R)}{\cos(\te_0) + \sin(\te_0)u_{+,\te_R}'(z,0)},   \\
& \Big(\Lates (z)\Big)_{2,2} = \frac{\cos(\te_R^{\prime})
- \sin(\te_R^{\prime})u_{-,\te_0}'(z,R)}{\cos(\te_R) - \sin(\te_R)u_{-,\te_0}'(z,R)}.
\lb{2.20a}
\end{split}
\end{align}
In particular, by \eqref{2.5a} and \eqref{2.5b},
\begin{equation}
\Big(\Lambda_{\te_0,\te_R}^{\te_0,\te_R'} (z)\Big)_{1,2} = 0, \quad 
\Big(\Lambda_{\te_0,\te_R}^{\te_0',\te_R} (z)\Big)_{2,1} = 0.
\end{equation} 
\end{remark}
%%%%%%%%%%%%%%%%%%%%%%%%%%%%%%%%%%%%%%%%

%%%%%%%%%%%%%%%
\begin{remark} \lb{r2.6}
We note that $\Lazzqq(z)$ represents the \emph{Dirichlet-to-Neumann map},
$\Lambda_{D,N}(z)$, for the boundary value problem \eqref{2.3}, \eqref{2.4}; that is,  when
$\te_0=\te_R=0$, $\te_0^{\prime}=\te_R^{\prime}=\pi/2$, then \eqref{2.7} becomes
\begin{equation}%\label{}
\Lambda_{D,N}(z) \begin{bmatrix}u(z,0)\\u(z,R)\end{bmatrix}=
\Lazzqq( z) \begin{bmatrix}u(z,0)\\u(z,R)\end{bmatrix}=
\lbrac\begin{array}{r}u'(z,0)\\-u'(z,R)\end{array}\rbrac, \quad z\in\bbC\backslash\sigma(\Hte),
\end{equation}
with $u(z,\cdot)=u(z,\cdot \, ;(0,c_0),(0,c_R))$, $u(z,0)=c_0$, $u(z,R)=c_R$. The Dirichlet-to-Neumann map in the case $V=0$ has recently been considered in
\cite[Example 5.1]{Po08}. The Neumann-to-Dirichlet map $\Lambda_{N,D}(z) 
= \Lambda_{\pi/2,\pi/2}^{\pi,\pi} (z) = - [\Lambda_{D,N}(z)]^{-1}$ (cf.\ \eqref{4.60})
in the case $V=0$ 
has earlier been computed in \cite[Example\ 4.1]{DM92}. We also refer to 
\cite{BMN08}, \cite{BGP08}, \cite{DM91}, \cite{FOP06} in the intimately related 
context of $Q$ and $M$-functions.
\end{remark}
%%%%%%%%%%%%%%%%

It would be interesting to establish precise connections between
$\La_{\te_0,\te_R}^{\te_0,\te_R} (z)$ and the dynamical response operator discussed, for instance, in \cite{ABI92}, \cite{AK08}, \cite{ALP02}, \cite{ALP05} in connection with the problem of regularity and controllability of the wave equation on a compact interval.

We conclude this section with an elementary result needed in the proof of
Lemma \ref{l3.4} and Theorem \ref{t4.6}:

%%%%%%%%%%%%%%%%
\begin{lemma} \lb{l2.7}
Let $V\in L^1((0,R); dx)$, fix $\te_0, \te_R\in S_{2 \pi}$, $c_0, c_R\in\bbC$, and assume that $z\in\CR$. Then the unique solution $u(z,\cdot)=u(z,\cdot\,;(\te_0,c_0),(\te_R,c_R))$ of \eqref{2.3}, \eqref{2.4} can be expressed in terms of the fundamental system
$\theta (z,\cdot), \phi (z,\cdot)$, as introduced in the proof of Lemma \ref{l2.1}, and
satisfying the initial conditions \eqref{2.2b}, as follows:
\begin{align}
& \, u(z,\cdot)=u(z,\cdot\,;(\te_0,c_0),(\te_R,c_R)) = \alpha (z) \theta (z,\cdot)
+ \beta (z) \phi (z,\cdot),    \lb{2.37} \\
& \begin{bmatrix} \alpha (z) \\[1mm] \beta (z) \end{bmatrix}
= \f{1}{\Delta (z,R,\te_0,\te_R)} \begin{bmatrix}
[\cos(\te_R) \phi (z,R) - \sin(\te_R) \phi' (z,R)] c_0 - \sin(\te_0) c_R \\[1mm]
- [\cos(\te_R) \theta (z,R) - \sin(\te_R) \theta' (z,R)] c_0 + \cos(\te_0) c_R \end{bmatrix},  \no
\end{align}
with $\Delta$ given by \eqref{2.2d}. In particular, one has
\begin{align}
\begin{split}
& \, u_{+,\te_R}(z,\cdot)=u(z,\cdot\,;((0,1),(\te_R,0)) = \alpha (z) \theta (z,\cdot)
+ \beta (z) \phi (z,\cdot),    \lb{2.38} \\
& \begin{bmatrix} \alpha (z) \\[1mm] \beta (z) \end{bmatrix}
= \f{1}{\Delta (z,R,0,\te_R)} \begin{bmatrix}
\cos(\te_R) \phi (z,R) - \sin(\te_R) \phi' (z,R)  \\[1mm]
- \cos(\te_R) \theta (z,R) + \sin(\te_R) \theta' (z,R) \end{bmatrix},
\end{split}
\intertext{and}
& \, u_{-,\te_0}(z,\cdot)=u(z,\cdot\,;(\te_0,0),(0,1))      \lb{2.39} \\
& \, \qquad \qquad \, = \f{1}{\Delta (z,R,\te_0,0)} \big(-\sin(\te_0) \theta (z,\cdot)
+ \cos(\te_0) \phi (z,\cdot)\big).   \no
\end{align}
\end{lemma}
%%%%%%%%%%%%%%%%
\begin{proof} With $u(z,\cdot)=u(z,\cdot\,;(\te_0,c_0),(\te_R,c_R))
= \alpha \theta (z,\cdot) + \beta \phi (z,\cdot)$ one observes from \eqref{2.4} that
\begin{align}
& \begin{bmatrix} c_0 \\ c_R \end{bmatrix} =
\gamma_{\te_0,\te_R} (u(z,\cdot)) = \gamma_{\te_0,\te_R} (\alpha \theta (z,\cdot)
+ \beta \phi (z,\cdot))   \no \\
& \quad = \begin{bmatrix}  \gamma_{\te_0,\te_R} (\theta (z,\cdot)) \quad
\gamma_{\te_0,\te_R} (\phi (z,\cdot)) \end{bmatrix} \begin{bmatrix} \alpha \\ \beta \end{bmatrix}
\no \\
& \quad = \begin{bmatrix} \cos(\te_0) & \sin(\te_0) \\[1mm]
\cos(\te_R) \theta (z,R) - \sin(\te_R) \theta' (z,R) & \cos(\te_R) \phi (z,R)
- \sin(\te_R) \phi' (z,R)
\end{bmatrix} \begin{bmatrix} \alpha \\[1mm] \beta \end{bmatrix}   \no \\
& \quad = \cU(z,R,\te_0,\te_R) \begin{bmatrix} \alpha \\ \beta \end{bmatrix},   \lb{2.40}
\end{align}
with $\cU$ introduced in \eqref{2.2ca}. Solving \eqref{2.40} for $\alpha, \beta$ (and observing that
$\det(\cU) = \Delta$ according to \eqref{2.2det}) then yields \eqref{2.37} and hence the special cases \eqref{2.38} and \eqref{2.39}.
\end{proof}
%%%%%%%%%%%%%%%%

%%%%%%%%%%%%%%%%%%%%%%%%%%%%%%%%%%%%%%%%
%%%%%%%%%%%%%%%%%%%%%%%%%%%%%%%%%%%%%%%%
\section{Resolvent Formulas}  \label{s3}
%%%%%%%%%%%%%%%%%%%%%%%%%%%%%%%%%%%%%%%%
%%%%%%%%%%%%%%%%%%%%%%%%%%%%%%%%%%%%%%%%

In the principal result of this section, we will derive a formula for $\Lates$ in terms of the resolvent of $\Hte$ and the boundary traces $\gamma_{\theta'_0,\theta'_R}$. But first we focus on the special boundary data map given by
\begin{equation}
\Late(z) = \Lateq (z), \quad \te_0,\te_R\in S_{2 \pi}, \quad z\in\bbC\backslash
\sigma(\Hte),       \lb{3.0}
\end{equation}
a generalization of the Dirichlet-to-Neumann map $\Lambda_{D,N}(z)=\Lazz(z)$. Using the basis for solutions of \eqref{2.3} given in \eqref{2.5} by
$u_{+,\te_R}(z,\cdot), u_{-,\te_0}(z,\cdot)$, and with
$\te_0^{\prime}=(\te_0+\pi/2) \, \text{\rm mod} (2 \pi)$,
$\te_R^{\prime}=(\te_R+\pi/2) \, \text{\rm mod} (2 \pi)$, equation \eqref{2.20} becomes
\begin{align}
& \Late (z) = \big[\big(\Late (z)\big)_{j,k}\big]_{1 \leq j,k \leq 2}, \quad
z\in\bbC\backslash\sigma(\Hte),  \label{3.1}   \\[1mm]
\begin{split}
& \big(\Late (z)\big)_{1,1} = \frac{-\sin(\te_0)
+ \cos(\te_0)u_{+,\te_R}'(z,0)}{\cos(\te_0) + \sin(\te_0)u_{+,\te_R}'(z,0)},    \\
& \big(\Late (z)\big)_{1,2} = \frac{-\sin(\te_0)u_{-,\te_0}(z,0)
+ \cos(\te_0)u_{-,\te_0}'(z,0)}{\cos(\te_R) - \sin(\te_R)u_{-,\te_0}'(z,R)},    \\
& \big(\Late (z)\big)_{2,1} = \frac{-\sin(\te_R)u_{+,\te_R}(z,R)
- \cos(\te_R)u_{+,\te_R}'(z,R)}{\cos(\te_0) + \sin(\te_0)u_{+,\te_R}'(z,0)},  \lb{3.1a} \\
& \big(\Late (z)\big)_{2,2} = \frac{-\sin(\te_R)
- \cos(\te_R)u_{-,\te_0}'(z,R)}{\cos(\te_R) - \sin(\te_R)u_{-,\te_0}'(z,R)}.
\end{split}
\end{align}

Next, consider 
\begin{equation}\label{3.2}
\al(\xi)=\begin{bmatrix}\cos(\xi) & \sin(\xi)\end{bmatrix},\quad \xi\in S_{2 \pi},
\end{equation}
and let $\Psi(z,\cdot \, ;x_0,\alpha(\xi))$ denote the fundamental matrix of solutions of
\eqref{2.3} given by
\begin{align}\label{3.3}
\Psi(z,\cdot \, ;x_0,\alpha(\xi))&=\begin{bmatrix}\Te(z,\cdot \, ;x_0,\alpha(\xi)) & \Phi(z,\cdot \, ;x_0,\alpha(\xi))
\end{bmatrix}  \no \\
&=\begin{bmatrix}\vartheta(z,\cdot \, ;x_0,\alpha(\xi))&\varphi(z,\cdot \, ;x_0,\alpha(\xi))\\ \vartheta'(z,\cdot \, ;x_0\alpha(\xi))&\varphi'(z,\cdot \, ;x_0,\alpha(\xi))\end{bmatrix}, \\
\intertext{in particular,} 
\Psi(z,x_0;x_0,\alpha(\xi))&=\lbrac\begin{array}{rr}\cos(\xi)&-\sin(\xi)\\ \sin(\xi)&\cos(\xi) \end{array}\rbrac,\quad x_0\in[0,R].   \no
\end{align}
In addition, set 
\begin{equation}
\beta(\eta)=\begin{bmatrix}\cos(\eta) & -\sin(\eta)\end{bmatrix}
= \al(-\eta), \quad \eta\in S_{2 \pi}.
\lb{3.3a}
\end{equation}
As shown in \cite[Section 2]{CG02} and \cite{CG06}, one can then introduce for all $x_0,y_0\in[0,R]$,
$x_0\ne y_0$, the \emph{Weyl--Titchmarsh function}, $m(z;x_0,\al(\xi);y_0,\beta(\eta))$, given by 
\begin{align}\label{3.4}
& m(z;x_0,\al(\xi);y_0,\beta(\eta))   \no \\
& \quad = -[\beta(\eta)\Phi(z,y_0;x_0,\al(\xi))]^{-1}[\beta(\eta)\Te(z,y_0;x_0,\al(\xi))]
\no \\[1mm]
& \quad =-\dfrac{\cos(\eta)\vartheta(z,y_0;x_0,\al(\xi)) - \sin(\eta)\vartheta'(z,y_0;x_0,\al(\xi))}{\cos(\eta)\varphi(z,y_0;x_0,\al(\xi)) - \sin(\eta)\varphi'(z,y_0;x_0,\al(\xi))},\\[1mm]
&\hspace*{4.6cm}  \xi,\eta\in S_{2 \pi}, \; z\in\bbC\backslash\sigma(H_{\xi,\eta}).  \no
\end{align}
Then, with $\psi(z,\cdot \, ;x_0,\al(\xi);y_0,\beta(\eta))$ defined by
\begin{align}\label{3.5}
\begin{split}
& \psi(z,\cdot \, ;x_0,\al(\xi);y_0,\beta(\eta))  \\
& \quad =
\vartheta(z,\cdot \, ;x_0,\al(\xi))+ \varphi(z,\cdot \, ;x_0,\al(\xi))m(z;x_0,\al(\xi);y_0,\beta(\eta)),
\end{split}
\end{align}
$\varphi(z,\cdot \, ;x_0,\al(\xi))$ and $\psi(z,\cdot \, ;x_0,\al(\xi);y_0,\beta(\eta))$ are linearly independent solutions of \eqref{2.3} satisfying
\begin{align}
\cos(\xi)\varphi(z,x_0;x_0,\al(\xi)) + \sin(\xi)\varphi'(z,x_0;x_0,\al(\xi))&=0,\label{3.6}\\
\cos(\eta)\psi(z,y_0;x_0,\al(\xi);y_0,\beta(\eta)) - \sin(\eta)\psi'(z,y_0;x_0,\al(\xi);y_0,\beta(\eta))&=0.   \label{3.7}
\end{align}

The following result has been proved in \cite[Lemma 2.10]{CG02} in a self-adjoint context, but self-adjointness is of no relevance for the result \eqref{3.8} below. For the convenience of the reader we reproduce its proof here.  

%%%%%%%%%%%%%%%
\begin{lemma} [\cite{CG02}, Lemma 2.10] \lb{l3.1}
Suppose $\xi, \eta \in S_{2 \pi}$,
$z\in\bbC\big\backslash\big(\sigma(H_{\xi,\eta}) \cup \sigma(H_{0,\eta})\big)$,
let $\al(\xi)$ and $\beta(\eta)$ be defined as  in \eqref{3.2} and \eqref{3.3a}, and
suppose that $x_0, y_0 \in [0,R]$.
Then, the following linear fractional transformation holds,
\begin{equation}\label{3.8}
m(z;x_0,\al(\xi);y_0,\beta(\eta))=\dfrac{-\sin(\xi) + \cos(\xi)m(z;x_0,\al(0);y_0,\beta(\eta))}
{\cos(\xi) + \sin(\xi)m(z;x_0,\al(0);y_0,\beta(\eta))},
\end{equation}
with $m$ defined as in \eqref{3.4}.
\end{lemma}
%%%%%%%%%%%%%%%
\begin{proof}
With $\psi(z,\cdot \, ;x_0,\al(\xi);y_0,\beta(\eta))$ defined in \eqref{3.5}, by \eqref{3.7} one obtains
\begin{equation}%\label{}
\beta(\eta)\begin{bmatrix}\psi(z,y_0;x_0,\al(0);y_0,\beta(\eta)) \\ \psi'(z,y_0;x_0,\al(0);y_0,\beta(\eta))  \end{bmatrix}=
\beta(\eta)\begin{bmatrix}\psi(z,y_0;x_0,\al(\xi);y_0,\beta(\eta)) \\ \psi'(z,y_0;x_0,\al(\xi);y_0,\beta(\eta))  \end{bmatrix}=0,
\end{equation}
and as a consequence, for some $c(z)\in \bbC\backslash\{0\}$,
\begin{equation}%\label{}
\begin{bmatrix}\psi(z,y_0;x_0,\al(\xi);y_0,\beta(\eta)) \\ \psi'(z,y_0;x_0,\al(\xi);y_0,\beta(\eta))  \end{bmatrix}=
c(z)\begin{bmatrix}\psi(z,y_0;x_0,\al(0);y_0,\beta(\eta)) \\ \psi'(z,y_0;x_0,\al(0);y_0,\beta(\eta))  \end{bmatrix}.
\end{equation}
By the uniqueness of solutions for the initial value problem associated with \eqref{2.3},
\begin{equation}\label{3.13}
\begin{bmatrix}\psi(z,x;x_0,\al(\xi);y_0,\beta(\eta)) \\ \psi'(z,x;x_0,\al(\xi);y_0,\beta(\eta))  \end{bmatrix}=
c(z)\begin{bmatrix}\psi(z,x;x_0,\al(0);y_0,\beta(\eta)) \\ \psi'(z,x;x_0,\al(0);y_0,\beta(\eta))  \end{bmatrix}, \quad
x\in{[0,R]}.
\end{equation}
Next, one notes that
\begin{align}\label{3.14}
&\begin{bmatrix}\psi(z,x_0;x_0,\al(\xi);y_0,\beta(\eta)) \\ \psi'(z,x_0;x_0,\al(\xi);y_0,\beta(\eta))  \end{bmatrix}=
\Psi(z,x_0;x_0,\al(\xi))\begin{bmatrix}1\\ m(z;x_0,\al(\xi);y_0,\beta(\eta)) \end{bmatrix}  \no \\
& \quad =\begin{bmatrix}\cos(\xi)& -\sin(\xi)\\ \sin(\xi)& \cos(\xi) \end{bmatrix}
\begin{bmatrix}1\\ m(z;x_0,\al(\xi);y_0,\beta(\eta)) \end{bmatrix},
\end{align}
and similarly that
\begin{align}\label{3.15}
\begin{bmatrix}\psi(z,x_0;x_0,\al(0);y_0,\beta(\eta)) \\ \psi'(z,x_0;x_0,\al(0);y_0,\beta(\eta))  \end{bmatrix}&=
\Psi(z,x_0;x_0,\al(0))\begin{bmatrix}1\\ m(z;x_0,\al(0);y_0,\beta(\eta)) \end{bmatrix}  \no \\
&=\begin{bmatrix}1\\ m(z;x_0,\al(0);y_0,\beta(\eta)) \end{bmatrix}.
\end{align}
By \eqref{3.13}--\eqref{3.15}, one concludes that
\begin{equation}%\label{}
\begin{bmatrix}1\\ m(z;x_0,\al(\xi);y_0,\beta(\eta)) \end{bmatrix}=
c(z)\begin{bmatrix}\cos(\xi)& \sin(\xi)\\ -\sin(\xi)& \cos(\xi) \end{bmatrix}
\begin{bmatrix}1\\ m(z;x_0,\al(0);y_0,\beta(\eta)) \end{bmatrix}
\end{equation}
implying \eqref{3.8}.
\end{proof}
%%%%%%%%%%%%%%%

With $\al(\xi)$ and $\beta(\eta)$ defined in \eqref{3.2} and \eqref{3.3a},
and with $m(z;x_0,\al(\xi);y_0,\beta(\eta))$ defined in \eqref{3.4}, for the next result we let
\begin{align}
&\al_0=\al(\te_0)=\begin{bmatrix}\cos(\te_0) & \sin(\te_0)
\end{bmatrix},\quad
 \beta_R=\beta(\te_R)=\begin{bmatrix}\cos(\te_R) & -\sin(\te_R)\end{bmatrix},\label{3.9}\\
&m_{+,\te_0}(z,\te_R)=m(z;0,\al_0;R,\beta_R),\quad 
m_{-,\te_R}(z,\te_0)=m(z;R,\beta_R; 0,\al_0).\label{3.10}
\end{align}

%%%%%%%%%%%%%%%%%%%%%%%%%%%%%%%%%%%
\begin{theorem}\label{t3.2}
Let $\te_0, \te_R \in S_{2 \pi}$ and $z\in\bbC\backslash\sigma(\Hte)$. Then,
\begin{equation}
\Late (z) = \begin{bmatrix} m_{+,\te_0}(z,\te_R) & \Late (z)_{1,2}  \\[1mm]
\Late (z)_{2,1} & -m_{-,\te_R}(z,\te_0) \end{bmatrix},    \lb{3.11}
\end{equation}
where
\begin{align}
&\Late (z)_{1,2} = \Late (z)_{2,1} \no \\[1mm]
\begin{split}
& \quad =  \dfrac{-\sin(\te_0)u_{-,\te_0}(z,0) + \cos(\te_0)u_{-,\te_0}'(z,0)}{\cos(\te_R) 
- \sin(\te_R)u_{-,\te_0}'(z,R)} \\[1mm]
& \quad = \dfrac{-\sin(\te_R)u_{+,\te_R}(z,R) - \cos(\te_R)u_{+,\te_R}'(z,R)}{\cos(\te_0) 
+ \sin(\te_0)u_{+,\te_R}'(z,0)}.    \lb{3.11a}
\end{split}
\end{align}
\end{theorem}
%%%%%%%%%%%%%%%%%%%%%%%%%%%%%%%%%%%
\begin{proof}
We temporarily assume in addition that
$z\notin \big(\sigma(H_{\theta_0,0}) \, \cup \, \sigma(H_{0,\theta_R})\big)$.
If $x_0=0$, $y_0=R$, then with $\varphi(z,\cdot \, ;0,\al_0)$ defined in \eqref{3.3}, and with $\psi(z,\cdot \, ;0,\al_0; R,\beta_R)$ defined in \eqref{3.5}, one notes, by \eqref{3.6} and  \eqref{3.7}, that
\begin{align}%\label{}
\varphi(z,\cdot \, ;0,\al_0) &= C_-(z) u_{-,\te_0}(z,\cdot),  \\
\psi(z,\cdot \, ;0,\al_0; R,\beta_R) &= C_+(z) u_{+,\te_R}(z,\cdot),
\end{align}
for some $C_{\pm}(z)\in\bbC\backslash\{0\}$, where
$u_{+,\te_R}(z,\cdot), u_{-,\te_0}(z,\cdot)$ represents the basis for the solutions of
\eqref{2.3} described in \eqref{2.5}. With  $m_{+,\te_0}(z,\te_R)$ defined in \eqref{3.10}, one notes when $\te_0=0$ that
\begin{equation}%\label{}
m_{+,0}(z,\te_R)
=\dfrac{\psi'(z,0;0,\al_0; R,\beta_R)}{\psi(z,0;0,\al_0; R,\beta_R)}\bigg|_{\te_0=0} 
= u_{+,\te_R}'(z,0);  \lb{3.11b}
\end{equation}
hence by \eqref{3.8} that
\begin{align}
\begin{split}
m_{+,\te_0}(z,\te_R) &= \dfrac{-\sin(\te_0) + \cos(\te_0)m_{+,0}(z,\te_R)}
{\cos(\te_0) + \sin(\te_0)m_{+,0}(z,\te_R)}  \\[1mm]
&= \dfrac{-\sin(\te_0) + \cos(\te_0)u_{+,\te_R}'(z,0)}{\cos(\te_0) 
+ \sin(\te_0)u_{+,\te_R}'(z,0)}.   \lb{3.24}
\end{split}
\end{align}
By Theorem \ref{t2.3}, $\Late$ is invariant with respect to a change of basis for
\eqref{2.3}; thus the $(1,1)$-entry of $\Late$, provided in  \eqref{3.1a}, equals 
$m_{+,\te_0}(z,\te_R)$.

Next, if $x_0=R$, $y_0=0$,  then with $\varphi(z,\cdot \, ;R,\beta_R)$ defined in 
\eqref{3.3}, and with $\psi(z,\cdot \, ;R,\beta_R;0,\al_0)$ defined in \eqref{3.5}, one notes, by \eqref{3.6} and \eqref{3.7}, that
\begin{align}%\label{}
\varphi(z,\cdot \, ;R,\beta_R) &= D_+(z) u_{+,\te_R}(z,\cdot), \\
\psi(z,\cdot \, ;R,\beta_R;0,\al_0) &= D_-(z) u_{-,\te_0}(z,\cdot),
\end{align}
for some $D_{\pm}(z)\in\bbC\backslash\{0\}$,
where  $u_{+,\te_R}(z,\cdot), u_{-,\te_0}(z,\cdot)$ again denotes the basis for the solutions of \eqref{2.3} described in \eqref{2.5}. With $m_{-,\te_R}(z,\te_0)$ defined in \eqref{3.10}, one now obtains when $\te_R=0$ that,
\begin{equation}%\label{}
m_{-,0}(z,\te_0)
=\dfrac{\psi'(z,R;R,\beta_R;0,\al_0)}{\psi(z,R;R,\beta_R;0,\al_0)}\bigg|_{\te_R=0} 
= u_{-,\te_0}'(z,R); \lb{3.25}
\end{equation}
hence by \eqref{3.8} that
\begin{align}
\begin{split}
m_{-,\te_R}(z,\te_0) &= \dfrac{\sin(\te_R) + \cos(\te_R)m_{-,0}(z,\te_0)}
{\cos(\te_R) - \sin(\te_R)m_{-,0}(z,\te_0)}   \\[1mm]
&= \dfrac{\sin(\te_R) + \cos(\te_R)u_{-,\te_0}'(z,R)}{\cos(\te_R)
- \sin(\te_R)u_{-,\te_0}'(z,R)}.   \lb{3.28}
\end{split}
\end{align}
Again, by the invariance of $\Late$ with respect to a change of basis for solutions of \eqref{2.3}, the $(2,2)$-entry of $\Late$, provided in \eqref{3.1a},  is  given by 
$-m_{-,\te_R}(z,\te_0)$.

To see that the off-diagonal elements of \eqref{3.11} are equal, we introduce
\begin{equation}\label{3.17}
W(z) = W(u_{+,\te_R}(z,\cdot), u_{-,\te_0}(z,\cdot)), \quad
\theta_0,\theta_R \in S_{2 \pi}, \; z\in\bbC\backslash\sigma(\Hte),
\end{equation}
where $u_{+,\te_R}(z,\cdot), u_{-,\te_0}(z,\cdot)$, is the basis for the solutions of
\eqref{2.3} as described in \eqref{2.5}. A straightforward computation, then yields
\begin{align}
& [-\sin(\te_0)u_{-,\te_0}(z,0) + \cos(\te_0)u_{-,\te_0}'(z,0)]
[\cos(\te_0) + \sin(\te_0)u_{+,\te_R}'(z,0)]   \no \\
& \quad = W(z)    \lb{3.18} \\
& \quad = [\cos(\te_R) - \sin(\te_R)u_{-,\te_0}'(z,R)]
[-\sin(\te_R)u_{+,\te_R}(z,R) - \cos(\te_R)u_{+,\te_R}'(z,R)].    \no
\end{align}
Indeed, the first equality in \eqref{3.18} follows by inserting the second term in 
\eqref{2.5a} into \eqref{2.5c}; similarly, the second equality in \eqref{3.18} follows by inserting the second term in \eqref{2.5b} into \eqref{2.5d}. Equation \eqref{3.18} 
immediately yields \eqref{3.11a}.

Finally, a meromorphic continuation with respect to $z$ then removes the additional assumption
$z\notin \big(\sigma(H_{\theta_0,0})\cup \sigma(H_{0,\theta_R})\big)$, completing the proof.
\end{proof}
%%%%%%%%%%%%%%%%%%%%%%%%%%%%%%%%%%

%%%%%%%%%%%%%
\begin{remark} \lb{r3.3}
Assume the special self-adjoint case where $V$ is real-valued and 
$\te_0, \te_R \in [0, 2\pi)$.  
By Marchenko's fundamental uniqueness result \cite[Ch.\ 2]{Ma73}, one 
observes the 
inverse spectral theory fact that each of the two diagonal terms of 
$\Late (\cdot)$ already uniquely determines the potential coefficient $V(x)$ 
for a.e.\ $x\in [0,R]$. In addition, the known leading asymptotic behavior of 
$m_{+,\te_0}(z,\te_R)$ as $z \to i \, \infty$, 
\begin{align}
\begin{split}
m_{+,\te_0}(z,\te_R)& \underset{z \to i \infty}{\longrightarrow} \cot(\te_0) + \oh(1),  
\quad \te_0 \in [0, 2\pi)\backslash\{0,\pi\},  \\
m_{+,0}(z,\te_R) & \underset{z \to i \infty}{\longrightarrow} i z^{1/2} + \oh\big(z^{1/2}\big),  
\quad \te_0 \in \{0,\pi\},
\end{split}
\end{align}
determines the boundary condition parameter $\te_0$, and similarly, that of 
$m_{-,\te_R}(z,\te_0)$ determines 
$\te_R$. These facts are obviously not shared by the usual $2 \times 2$ 
matrix-valued Weyl--Titchmarsh function to be briefly discussed in our final 
Section \ref{s7} and hence demonstrates a stark contrast between these two matrix-valued functions, even though, both are Herglotz functions. The 
Herglotz property is well-known to be an intrinsic property of the $2 \times 2$ 
matrix-valued Weyl--Titchmarsh function (cf., e.g., \cite[Sect.\ 9.5]{CL85}), 
\cite[Sect.\ II.8]{LS75}, \cite[Sect.\ 6.5]{Pe88}, 
\cite[Ch.\ III]{Ti62}) and will also be verified for $\Late (\cdot)$ as a special 
case of the principal result in Theorem \ref{t4.6} (cf.\  \eqref{4.76}).
\end{remark}
%%%%%%%%%%%%%%

Next we turn to the asymptotic behavior of $\Late (z)$ as $|z|\to\infty$,
$\Im(z^{1/2})>0$:

%%%%%%%%%%%%%%
\begin{lemma} \lb{l3.4}
Suppose that $V\in L^1((0,R); dx)$ and let $\te_0, \te_R\in S_{2 \pi}$. Then,
\begin{align}
& \Late (z) \underset{\substack{|z|\to\infty\\ \Im(z^{1/2})>0}}{=}
\begin{bmatrix} \cot(\te_0) & \f{2i  z^{-1/2} e^{i z^{1/2} R}}{\sin(\te_0)\sin(\te_R)} \\
\f{2i z^{-1/2} e^{i z^{1/2} R}}{\sin(\te_0)\sin(\te_R)}
& - \cot(\te_0) \end{bmatrix}  \no \\[1mm]
& \quad + \begin{bmatrix} \Oh(|z|^{-1/2} ) 
& \Oh\big(|z|^{-1} e^{-\Im(z^{1/2}) R}\big) \\
\Oh\big(|z|^{-1} e^{-\Im(z^{1/2}) R}\big) & \Oh(|z|^{-1/2}) \end{bmatrix},
\quad \te_0 \neq 0, \, \te_R \neq 0, \\
& \Lambda_{0,\theta_R} (z) \underset{\substack{|z|\to\infty\\ \Im(z^{1/2})>0}}{=}
\begin{bmatrix} i z^{1/2} & - \f{2 e^{i z^{1/2} R}}{\sin(\te_R)} \\
- \f{2 e^{i z^{1/2} R}}{\sin(\te_R)}
& - \cot(\te_R) \end{bmatrix}  \no \\[1mm]
& \quad + \begin{bmatrix} \Oh(1) & \Oh\big(|z|^{-1/2} e^{-\Im(z^{1/2}) R}\big) \\
\Oh\big(|z|^{-1/2} e^{-\Im(z^{1/2}) R}\big) & \Oh(|z|^{-1/2}) \end{bmatrix},
\quad \te_0 = 0, \, \te_R \neq 0, \\
& \Lambda_{\theta_0,0} (z) \underset{\substack{|z|\to\infty\\ \Im(z^{1/2})>0}}{=}
\begin{bmatrix} \cot(\te_0) & - \f{2 e^{i z^{1/2} R}}{\sin(\te_0)} \\
- \f{2 e^{i z^{1/2} R}}{\sin(\te_0)}
& - i z^{1/2} \end{bmatrix}  \no \\[1mm]
& \quad + \begin{bmatrix} \Oh(|z|^{-1/2}) 
& \Oh\big(|z|^{-1/2} e^{-\Im(z^{1/2}) R}\big) \\
\Oh\big(|z|^{-1/2} e^{-\Im(z^{1/2}) R}\big) & \Oh(1) \end{bmatrix},
\quad \te_0 \neq 0, \, \te_R = 0, \\
& \Lambda_{0,0} (z) \underset{\substack{|z|\to\infty\\ \Im(z^{1/2})>0}}{=}
\begin{bmatrix} i z^{1/2} & - 2i z^{1/2} e^{i z^{1/2} R} \\
- 2i z^{1/2} e^{i z^{1/2} R} & - i z^{1/2} \end{bmatrix}  \no \\[1mm]
& \quad + \begin{bmatrix} \Oh(1) & \Oh\big(e^{-\Im(z^{1/2}) R}\big) \\
\Oh\big(e^{-\Im(z^{1/2}) R}\big) & \Oh(1) \end{bmatrix},
\quad \te_0 = 0, \, \te_R = 0.
\end{align}
\end{lemma}
%%%%%%%%%%%%%%
\begin{proof}
This follows from an elementary computation upon inserting \eqref{2.38} (resp.,
\eqref{2.39}) into \eqref{3.24} (resp., \eqref{3.28}) and using the asymptotic
expansions \eqref{2.2e}.
\end{proof}
%%%%%%%%%%%%%%

For the principal result of this section, an explicit formula for $\Lates(z)$ in terms of the
resolvent $(\Hte- z I \big)^{-1}$ of $\Hte$ and the boundary traces
$\gamma_{\theta'_0,\theta'_R}$, we recall the Green's function associated with the operator $\Hte$ in \eqref{2.2a},
\begin{align}
& \Gte(z,x,x') = (\Hte -z I)^{-1}(x,x')   \no \\
& \quad = \frac{1}{W(u_{+,\te_R}(z,\cdot), u_{-,\te_0}(z,\cdot))} \begin{cases}
u_{-,\te_0}(z,x')u_{+,\te_R}(z,x), & 0\le x'\le x,\\[1mm]
u_{-,\te_0}(z,x)u_{+,\te_R}(z,x'), & 0\le x\le x',
\end{cases}   \label{3.32} \\[1mm]
&\hspace*{5.9cm}  z\in\bbC\backslash\sigma(\Hte), \; x, x' \in {[0,R]}.  \no
\end{align}
Here $u_{+,\te_R}(z,\cdot), u_{-,\te_0}(z,\cdot)$ is a basis for solutions of \eqref{2.3} as described in \eqref{2.5} and $I = I_{L^2((0,R); dx)}$ abbreviates the identity operator in $L^2((0,R); dx)$. Thus, one obtains
\begin{align}
\begin{split}
 \big((\Hte- z I)^{-1}g\big)(x) = \int_0^R \, dx'\, \Gte(z,x,x')g(x'),& \\
 g\in L^2((0,R); dx), \; z\in\bbC\backslash\sigma(\Hte), \; x\in(0,R).&  \lb{3.33}
\end{split}
\end{align}

Equation \eqref{3.32} is a crucial input in the proof of the following fundamental
result:

%%%%%%%%%%%%%%%%%%%%%%%%%%%%%%%%%%%
\begin{theorem} \lb{t3.5}
Assume that $\te_0, \te_R, \te_0', \te_R' \in S_{2 \pi}$, let $\Hte$ be defined as in
\eqref{2.2a}, suppose that $z\in\bbC\backslash\sigma(\Hte)$, and let
$S_{\theta_0' -\theta_0,\theta_R' -\theta_R}$ be defined according to \eqref{2.14}. 
Then 
\begin{equation}
 \Lates (z) S_{\theta_0' -\theta_0,\theta_R'-\theta_R}
=\gamma_{\te_0',\te_R'}
\big[\gamma_{\ol{\theta_0'},\ol{\theta_R'}} ((\Hte)^* - {\ol z} I)^{-1}\big]^*.
\lb{3.33aa}
\end{equation}
In particular, with
$\te_0^{\prime}=(\te_0+\pi/2) \, \text{\rm mod} (2 \pi)$,
$\te_R^{\prime}=(\te_R+\pi/2) \, \text{\rm mod} (2 \pi)$, one obtains
\begin{equation}
\Late (z) =\tgate \big[\hatt \gamma_{\ol{\theta_0},\ol{\theta_R}}((\Hte)^* - {\ol z} I)^{-1}\big]^*,
\lb{3.33a}
\end{equation}
where
\begin{equation}
\tgate=\gateq.
\end{equation}
As a consequence,
\begin{equation}
\Lambda_{(\theta_0 + \pi) \, \text{\rm mod} (2\pi),
(\theta_R + \pi) \, \text{\rm mod} (2\pi)} (z)
= \Late (z),  \quad  \te_0, \te_R\in S_{2 \pi}.     \lb{3.33A}
\end{equation}
\end{theorem}
%%%%%%%%%%%%%%%%%%%%%%%%%%%%%%%%%%%
\begin{proof}
We start by noting that the Green's function $G^*_{\te_0,\te_R}(z,x,x')$ of $(\Hte)^*$ is given by
\begin{align}\label{3.32a}
\begin{split}
\Gte^*(z,x,x') &= \big((\Hte)^* -z I \big)^{-1}(x,x')   \no \\
& = \frac{1}{W^*(z)} \begin{cases}
\ol{u_{-,\te_0}(\ol z,x')} \, \ol{u_{+,\te_R}(\ol z,x)}, & 0\le x'\le x,\\[2mm]
\ol{u_{-,\te_0}(\ol z,x)}\, \ol{u_{+,\te_R}(\ol z,x')}, & 0\le x\le x',
\end{cases}  \\[1mm]
&\hspace*{2.25cm}  z\in\bbC\backslash\sigma\big((\Hte)^*\big), \; x, x' \in {[0,R]},
\end{split}
\end{align}
where $u_{+,\te_R}(z,\cdot), u_{-,\te_0}(z,\cdot)$ is a basis for solutions of \eqref{2.3} as described in \eqref{2.5}, and $W^*(z)$ is defined by
\begin{equation}\label{3.17a}
W^*(z) = W\big(\ol{u_{+,\te_R}(\ol z,\cdot)}, \ol{u_{-,\te_0}(\ol z,\cdot)}\big),
\quad z\in\bbC\backslash\sigma(\Hte).
\end{equation}
In particular, we note the fact that
\begin{equation}\label{3.38}
\ol{W^*(\ol z)} = W(z),
\end{equation}
with $W(z)$ the Wronskian defined in \eqref{3.17}. Thus, one obtains
\begin{align}
\begin{split}
 \big(((\Hte)^*- z I)^{-1}g\big)(x) = \int_0^R \, dx'\, \Gte^*(z,x,x') g(x'),& \\
 g\in L^2((0,R); dx), \; z\in\bbC\backslash\sigma(\Hte), \; x\in(0,R).&  \lb{3.38a}
\end{split}
\end{align}
Next, for $g\in L^2((0,R); dx)$ let
\begin{align}%\label{}
\hg(\ol z,x) & = \big(((\Hte)^*- {\ol z} I)^{-1} g\big)(x)  \no  \\
& = \frac{1}{\ol{W(z)}} \bigg(\ol{u_{+,\te_R}(z,x)}
\int_0^x\, dx'\, \ol{u_{-,\te_0}(z,x')} g(x')     \\
& \hspace*{1.75cm} +  \ol{u_{-,\te_0}(z,x)} \int_x^R\, dx'\,
\ol{u_{+,\te_R}(z,x')} g(x')\bigg).   \no \\
\intertext{Then one notes that}
\begin{split}
\hg'(\ol z,x) & = \frac{1}{\ol{W(z)}} \bigg(\ol{u_{+,\te_R}'(z,x)} \int_0^x\, dx'\, \ol{u_{-,\te_0}(z,x')} g(x')  \\
& \hspace*{1.75cm} +  \ol{u_{-,\te_0}'(z,x)} \int_x^R\, dx'\, \ol{u_{+,\te_R}(z,x')} g(x')\bigg),
\end{split}
\end{align}
and, as a consequence, that
\begin{align}
\hg(\ol z,0)&=\frac{1}{\ol{W(z)}} \ol{u_{-,\te_0}(z,0)}
\int_0^R \, dx'\,  \ol{u_{+,\te_R}(z,x')} g(x'),  \label{3.36} \\
\hg'(\ol z,0)&=\frac{1}{\ol{W(z)}} \ol{u_{-,\te_0}'(z,0)}
\int_0^R \, dx'\,  \ol{u_{+,\te_R}(z,x')} g(x'),  \label{3.36a} \\
\hg(\ol z,R)& =\frac{1}{\ol{W(z)}} \ol{u_{+,\te_R}(z,R)}
\int_0^R \, dx'\,  \ol{u_{-,\te_0}(z,x')} g(x'),  \label{3.37} \\
\hg'(\ol z,R)&=\frac{1}{\ol{W(z)}} \ol{u_{+,\te_R}'(z,R)}
\int_0^R \, dx'\,  \ol{u_{-,\te_0}(z,x')} g(x'). \label{3.37a}
\end{align}
In turn, this permits one to compute 
\begin{align}
\begin{split}
\gamma_{\ol{\te_0'},\ol{\te_R'}} ((\Hte)^* - {\ol z} I)^{-1} g
&= \gamma_{\ol{\te_0'},\ol{\te_R'}} \hg (\bz,\cdot)   \\
& = \begin{bmatrix} \cos(\ol{\theta_0'}) \hatt g (\ol z, 0) + \sin(\ol{\theta_0'})
{\hatt g}^{\prime}(\ol z, 0) \\[1mm]
\cos(\ol{\theta_R'}) \hatt g(\ol z, R) - \sin(\ol{\theta_R'}) {\hatt g}^{\prime}(\ol z, R)
\end{bmatrix}
\label{3.42}.
\end{split}
\end{align}
Using \eqref{3.36}--\eqref{3.42} one infers, with
$[a_0 \;\, a_R]^\top \in \bbC^2$ and
$(\cdot,\cdot)_{\bbC^2}$ denoting the scalar product in $\bbC^2$, that
\begin{align}
&\big(\gamma_{\ol{\te_0'},\ol{\te_R'}} ((\Hte)^* - {\ol z} I)^{-1} g,
[a_0 \;\, a_R]^\top \big)_{\bbC^2} \no \\
& \quad = \left(
\begin{bmatrix} \cos(\ol{\theta_0'}) \hatt g (\ol z, 0) + \sin(\ol{\theta_0'})
{\hatt g}^{\prime}(\ol z, 0) \\[1mm]
\cos(\ol{\theta_R'}) \hatt g(\ol z, R) - \sin(\ol{\theta_R'}) {\hatt g}^{\prime}(\ol z, R)
\end{bmatrix}, \begin{bmatrix} a_0 \\[1mm] a_R \end{bmatrix} \right)_{\bbC^2} \no \\
& \quad =\frac{1}{W(z)} \int_0^R \, dx' \, \ol{g(x')}
\big(\cos(\theta_0') a_0 u_{-,\te_0}(z,0) u_{+,\te_R}(z,x')   \no \\
& \hspace*{1.3cm} + \sin(\theta_0') a_0 u_{-,\te_0}^{\prime}(z,0) u_{+,\te_R}(z,x')
 +\cos(\theta_R') a_R u_{+,\te_R}(z,R) u_{-,\te_0}(z,x')  \no \\
& \hspace*{1.3cm}  -\sin(\theta_R') a_R u_{+,\te_R}^{\prime}(z,R) u_{-,\te_0}(z,x') \big).
\end{align}
Hence one concludes that
\begin{align}
&\Big(\big[\gamma_{\ol{\te_0'},\ol{\te_R'}} ((\Hte)^*- {\ol z} I)^{-1}\big]^*
[a_0 \;\, a_R]^\top \Big)(x) \notag \\
&\quad =\frac{1}{W(z)} \Big(\big[\cos(\theta_0') u_{-,\te_0}(z,0)
+ \sin(\theta_0') u_{-,\te_0}^{\prime}(z,0)\big] a_0 u_{+,\te_R}(z,x)   \no \\
& \hspace*{1.95cm} + \big[\cos(\theta_R') u_{+,\te_R}(z,R)
- \sin(\theta_R') u_{+,\te_R}^{\prime}(z,R)\big] a_R u_{-,\te_0}(z,x) \Big).   \lb{3.43}
\end{align}
Consequently, writing
\begin{align}
\begin{split}
& \gamma_{\te_0',\te_R'} \big[\gamma_{\ol{\te_0'},\ol{\te_R'}}
((\Hte)^*- {\ol z} I)^{-1}\big]^* [a_0 \;\, a_R]^\top    \\
& \quad =
\begin{bmatrix}
\big(\gamma_{\te_0',\te_R'} \big[\gamma_{\ol{\te_0'},\ol{\te_R'}}
((\Hte)^*- {\ol z} I)^{-1}\big]^* [a_0 \;\, a_R]^\top\big)_1 \\[1mm]
\big(\gamma_{\te_0',\te_R'} \big[\gamma_{\ol{\te_0'},\ol{\te_R'}}
((\Hte)^*- {\ol z} I)^{-1}\big]^* [a_0 \;\, a_R]^\top\big)_2
\end{bmatrix},
\end{split}
\end{align}
it follows that
\begin{align}
& \big(\gamma_{\te_0',\te_R'} \big[\gamma_{\ol{\te_0'},\ol{\te_R'}}
((\Hte)^*- {\ol z} I)^{-1}\big]^* [a_0 \;\, a_R]^\top \big)_1   \no \\
& \quad = \f{1}{W(z)} \Big(\cos(\te_0')\big[\cos(\te_0') u_{-,\te_0}(z,0)
+ \sin(\te_0') u'_{-,\te_0}(z,0)\big] a_0   \no \\
& \qquad + \cos(\te_0')\big[\cos(\te_R') u_{+,\te_R}(z,R)
- \sin(\te_R') u'_{+,\te_R}(z,R)\big] a_R u_{-,\te_0}(z,0)   \no \\
& \qquad + \sin(\te_0')\big[\cos(\te_0') u_{-,\te_0}(z,0)
+ \sin(\te_0') u'_{-,\te_0}(z,0)\big] a_0 u'_{+,\te_R}(z,0)   \no \\
& \qquad + \sin(\te_0')\big[\cos(\te_R') u_{+,\te_R}(z,R)
- \sin(\te_R') u'_{+,\te_R}(z,R)\big] a_R u'_{-,\te_0}(z,0)\Big) ,    \\
\intertext{and} 
& \big(\gamma_{\te_0',\te_R'} \big[\gamma_{\ol{\te_0'},\ol{\te_R'}}
((\Hte)^*- {\ol z} I)^{-1}\big]^* [a_0 \;\, a_R]^\top \big)_2   \no \\
& \quad = \f{1}{W(z)} \Big(\cos(\te_R')\big[\cos(\te_0') u_{-,\te_0}(z,0)
+ \sin(\te_0') u'_{-,\te_0}(z,0)\big] a_0 u_{+,\te_R}(z,R)  \no \\
& \qquad + \cos(\te_R')\big[\cos(\te_R') u_{+,\te_R}(z,R)
- \sin(\te_R') u'_{+,\te_R}(z,R)\big] a_R    \no \\
& \qquad - \sin(\te_R')\big[\cos(\te_0') u_{-,\te_0}(z,0)
+ \sin(\te_0') u'_{-,\te_0}(z,0)\big] a_0 u'_{+,\te_R}(z,R)   \no \\
& \qquad - \sin(\te_R')\big[\cos(\te_R') u_{+,\te_R}(z,R)
- \sin(\te_R') u'_{+,\te_R}(z,R)\big] a_R u'_{-,\te_0}(z,R)\Big).
\end{align}
Hence, writing
\begin{align}
\begin{split}
& \gamma_{\te_0',\te_R'} \big[\gamma_{\ol{\te_0'},\ol{\te_R'}}
((\Hte)^*- {\ol z} I)^{-1}\big]^*    \\
& \quad = \Big[\big(\gamma_{\te_0',\te_R'} \big[\gamma_{\ol{\te_0'},\ol{\te_R'}}
((\Hte)^*- {\ol z} I)^{-1}\big]^*\big)_{j,k}\Big]_{j,k=1,2},
\end{split}
\end{align}
one obtains
\begin{align}
\begin{split}
& \big(\gamma_{\te_0',\te_R'} \big[\gamma_{\ol{\te_0'},\ol{\te_R'}}
((\Hte)^*- {\ol z} I)^{-1}\big]^*\big)_{1,1}    \\
& \quad = \f{1}{W(z)} \big[\cos(\te_0') u_{-,\te_0}(z,0)
+ \sin(\te_0') u'_{-,\te_0}(z,0)\big]    \\
& \qquad \qquad \quad \times \big[\cos(\te'_0)
+ \sin(\te'_0) u'_{+,\te_R}(z,0)\big],    \\
& \big(\gamma_{\te_0',\te_R'} \big[\gamma_{\ol{\te_0'},\ol{\te_R'}}
((\Hte)^*- {\ol z} I \big)^{-1}\big]^*\big)_{1,2}    \\
& \quad = \f{1}{W(z)} \big[\cos(\te_R') u_{+,\te_R}(z,R)
- \sin(\te_R') u'_{+,\te_R}(z,R)\big]    \\
& \qquad \qquad \quad \times \big[\cos(\te'_0) u_{-,\te_0}(z,0)
+ \sin(\te'_0) u'_{-,\te_0}(z,0)\big],    \\
& \big(\gamma_{\te_0',\te_R'} \big[\gamma_{\ol{\te_0'},\ol{\te_R'}}
((\Hte)^*- {\ol z} I)^{-1}\big]^*\big)_{2,1}    \\
& \quad = \f{1}{W(z)} \big[\cos(\te_0') u_{-,\te_0}(z,0)
+ \sin(\te_0') u'_{-,\te_0}(z,0)\big]     \\
& \qquad \qquad \quad \times \big[\cos(\te_R') u_{+,\te_R}(z,R)
- \sin(\te_R') u'_{+,\te_R}(z,R)\big],   \\
& \big(\gamma_{\te_0',\te_R'} \big[\gamma_{\ol{\te_0'},\ol{\te_R'}}
((\Hte)^*- {\ol z} I)^{-1}\big]^*\big)_{2,2}    \\
& \quad = \f{1}{W(z)} \big[\cos(\te_R') u_{+,\te_R}(z,R)
- \sin(\te_R') u'_{+,\te_R}(z,R)\big]     \\
& \qquad \qquad \quad \times \big[\cos(\te'_R)
- \sin(\te'_R) u'_{-,\te_0}(z,R)\big].    \lb{3.47}
\end{split}
\end{align}
Employing \eqref{2.5c} and \eqref{2.5d} one finally concludes from \eqref{2.20a} 
that the expressions in \eqref{3.47} are equivalent to the following ones
\begin{align}
\begin{split}
& \big(\gamma_{\te_0',\te_R'} \big[\gamma_{\ol{\te_0'},\ol{\te_R'}}
((\Hte)^*- {\ol z} I)^{-1}\big]^*\big)_{1,1}
= \sin(\te_0'-\te_0) \Big(\Lates (z)\Big)_{1,1},    \\
& \big(\gamma_{\te_0',\te_R'} \big[\gamma_{\ol{\te_0'},\ol{\te_R'}}
((\Hte)^*- {\ol z} I)^{-1}\big]^*\big)_{1,2}
= \sin(\te_R'-\te_R) \Big(\Lates (z)\Big)_{1,2},     \\
& \big(\gamma_{\te_0',\te_R'} \big[\gamma_{\ol{\te_0'},\ol{\te_R'}}
((\Hte)^*- {\ol z} I)^{-1}\big]^*\big)_{2,1}
= \sin(\te_0'-\te_0) \Big(\Lates (z)\Big)_{2,1},    \\
& \big(\gamma_{\te_0',\te_R'} \big[\gamma_{\ol{\te_0'},\ol{\te_R'}}
((\Hte)^*- {\ol z} I)^{-1}\big]^*\big)_{2,2}
= \sin(\te_R'-\te_R) \Big(\Lates (z)\Big)_{2,2},   \lb{3.50}
\end{split}
\end{align}
proving \eqref{3.33aa}.

Finally, equation \eqref{3.33A} is a consequence of \eqref{2.2A} and \eqref{3.33a}.
\end{proof}
%%%%%%%%%%%%%%%%%%%%%%%%%%%%%%%%%%%%%%%%

%%%%%%%%%%%%%
\begin{remark} \lb{r3.6}
A formula of the type \eqref{3.33a} for Dirichlet-to-Neumann maps associated with
multi-dimensional Schr\"odinger operators was published by Amrein and Pearson
\cite{AP04} in 2004. It has recently been extended in various directions in \cite{GM08},
\cite{GM09}, \cite{GM10}, \cite{GMZ07}.
Formula \eqref{3.33a} for the Dirichlet-to-Neumann map $\Lambda_{\pi/2,\pi/2}$ in the special case $V=0$ has also been derived by Posilicano \cite[Example 5.1]{Po08}.
\end{remark}
%%%%%%%%%%%%%%

%%%%%%%%%%%%%%
\begin{remark} \lb{r3.7}
While it is tempting to view $\gate$ as an unbounded but densely
defined operator on $L^2((0,R); dx)$ whose domain contains the space
$C_0^\infty((0,R))$, one should note, in this case, that its adjoint
$\gate^*$ is not densely defined. Indeed, the adjoint $\gate^*$ of $\gate$
would have to be an unbounded operator from $\bbC^2$
to $L^2((0,R); dx)$ such that
\begin{equation}
(\gate f,g)_{\bbC^2} = (f,\gate^* g)_{L^2((0,R); dx)}
\, \text{ for all }\, f\in\dom(\gate),\; g\in\dom(\gate^*).  \lb{3.55}
\end{equation}
In particular, choosing $f\in C_0^\infty((0,R))$, in which case $\gate f =0$,
one concludes that
\begin{equation}
(f,\gate^* g)_{L^2((0,R); dx)}=0 \, \text{ for
all } \, f\in C_0^\infty((0,R)).
\end{equation}
Thus, one obtains $\gate^* g = 0$ for all
$g\in\dom(\gate^*)$. Since obviously $\gate \neq 0$, \eqref{3.55}
implies $\dom(\gate^*)=\{0\}$ and hence $\gate$ is not a closable
linear operator in $L^2((0,R); dx)$. This is the reason for our careful choice of
notation in \eqref{3.33aa} and \eqref{3.33a}.
\end{remark}
%%%%%%%%%%%%%%

We conclude this section by providing an explicit example in the special case $V=0$ a.e.\ on
$[0,R]$. For notational purposes we add the superscript $(0)$ to the corresponding quantities below:

%%%%%%%%%%%%%%
\begin{example}\lb{e3.8}
Let $V=0$ a.e.\ on $[0,R]$, $x_0 \in (0,R)$,  $\te_0, \te_R, \te \in S_{2 \pi}$,
$x, x', s \in [0,R]$, and let
\begin{align}
&f(z,s,\al,\be)=f(z,s,\be,\al)  \\
&\quad
=z\sin(\al)\sin(\be)\sin(\sqrt{z}s) +\sqrt{z}\sin(\al+\be)\cos(\sqrt{z}s) - \cos(\al)\cos(\be)\sin(\sqrt{z}s), \no \\
&g(z,s,\al,\be)=f(z,s,\al+\pi/2,\be)  \\
&\quad
=  z\cos(\al)\sin(\be)\sin(\sqrt{z}s) +\sqrt{z}\cos(\al+\be)\cos(\sqrt{z}s) + \sin(\al)\cos(\be)\sin(\sqrt{z}s).  \no
\end{align}
Then,
\begin{align}
&u^{(0)} (z,x,(\te_0,c_0),(\te_R,c_R)) = \frac{c_0f(z,R-x,0,\te_R)
+ c_Rf(z,x,\te_0,0)}{f(z,R,\te_0,\te_R)},
\\
&u_{+,\te_R}^{(0)} (z,x) = u(z,\cdot \, ;(0,1),(\te_R,0))
= \frac{f(z,R-x,0,\te_R)}{f(z,R,0,\te_R)},
\\
&u_{+,\te_R}^{(0)\;'} (z,x)= \frac{f'(z,R-x,0,\te_R)}{f(z,R,0,\te_R)}
= \frac{f(z,R-x,\pi/2,\te_R)}{f(z,R,0,\te_R)},
\\
&u_{-,\te_0}^{(0)} (z,x) = u(z,\cdot \, ;(\te_0,0),(0,1))
= \frac{f(z,x,\te_0,0)}{f(z,R,\te_0,0)},
\\
&u_{-,\te_0}^{(0)\; '}(z,x) = \frac{f'(z,x,\te_0,0)}{f(z,R,\te_0,0)}
= -\frac{f(z,x,\te_0,\pi/2)}{f(z,R,\te_0,0)},
\\
&W(u^{(0)}_{+,\te_R}(z,\cdot), u^{(0)}_{-,\te_0}(z,\cdot))
= \frac{-\sqrt{z}f(z,R,\te_0,\te_R)}{f(z,R,\te_0,0)f(z,R,0,\te_R)},
\\
& m_{+,\te_0}^{(0)} (z,\te_R) = \frac{g(z,R,\te_0,\te_R)}{f(z,R,\te_0,\te_R)},
\\
& m_{-,\te_R}^{(0)} (z,\te_0) = - \frac{g(z,R,\te_R,\te_0)}{f(z,R,\te_R,\te_0)},
\\
& \Lambda_{\te_0,\te_R}^{(0)} (z) = \begin{pmatrix}
\frac{g(z,R,\te_0,\te_R)}{f(z,R,\te_0,\te_R)} & \frac{-\sqrt{z}}{f(z,R,\te_0,\te_R)} \\
\frac{-\sqrt{z}}{f(z,R,\te_0,\te_R)} & \frac{g(z,R,\te_R,\te_0)}{f(z,R,\te_0,\te_R)}
\end{pmatrix},   \lb{3.72} 
\no \\
&m_{+,0}^{(0)} (z, x_0,\te_R) = \f{u_{+,\te_R}^{(0)\;'} (z,x_0)}{u_{+,\te_R}^{(0)} (z,x_0)}
= \frac{g(z,R-x_0,0,\te_R)}{f(z,R-x_0,0,\te_R)},
\\
&m_{-,0}^{(0)} (z, x_0,\te_0) = \f{u_{-,\te_0}^{(0)\; '}(z,x_0)}{u_{-,\te_0}^{(0)}(z,x_0)}
= -\frac{g(z,x_0,0,\te_0)}{f(z,x_0,0,\te_0)},
\\
&G_{\te_0,\te_R}^{(0)} (z,x,x') = \big(\Hte^{(0)} -z I\big)^{-1}(x,x')  \no \\
&\quad = \frac{1}{W(u^{(0)}_{+,\te_R}(z,\cdot), u^{(0)}_{-,\te_0}(z,\cdot))}
\begin{cases}
u^{(0)}_{-,\te_0}(z,x')u^{(0)}_{+,\te_R}(z,x), & 0\le x'\le x,\\[1mm]
u^{(0)}_{-,\te_0}(z,x)u^{(0)}_{+,\te_R}(z,x'), & 0\le x\le x',
\end{cases}  \no \\[10pt]
 &\quad = \begin{cases}
\dfrac{f(z,x',\te_0,0)f(z,R-x,0,\te_R)}{-\sqrt{z}f(z,R,\te_0,\te_R)},&  0\le x'\le x,\\[10pt]
\dfrac{f(z,x,\te_0,0)f(z,R-x',0,\te_R)}{-\sqrt{z}f(z,R,\te_0,\te_R)},&  0\le x\le x',
\end{cases}  \quad
z\in\bbC\backslash\sigma\big(\Hte^{(0)}\big).
\end{align}
\end{example}
%%%%%%%%%%%%%%

The special case of the Neumann-to-Dirichlet map, $\Lambda^{(0)}_{\pi/2,\pi/2}$ in \eqref{3.72}, was computed in \cite[Example\ 4.1]{DM92}, and more recently, the
special case of the Dirichlet-to-Neumann map, $\Lambda^{(0)}_{0,0}$ in \eqref{3.72}, was computed in \cite[Example\ 5.1]{Po08}.

%%%%%%%%%%%%%%%%%%%%%%%%%%%%%%%%%%%%%%%%
%%%%%%%%%%%%%%%%%%%%%%%%%%%%%%%%%%%%%%%%
\section{Linear Fractional Transformations and the Herglotz Property in the 
Self-Adjoint Case}  \label{s4}
%%%%%%%%%%%%%%%%%%%%%%%%%%%%%%%%%%%%%%%%
%%%%%%%%%%%%%%%%%%%%%%%%%%%%%%%%%%%%%%%%

The principal purpose of this section is to prove that $\Lates(z)$ and $\Lades(z)$ 
satisfy a linear fractional transformation. As a consequence we will show that 
$\Lates(z) S_{\te_0'-\te_0,\te_R'-\te_R}$ is a $2 \times 2$ matrix-valued Herglotz 
function in the special self-adjoint case where $V$ and 
$\theta_0, \theta_R, \theta_0', \theta_R'$ are real-valued.

In the following we denote by $\bbC^{n \times n}$, $n\in\bbN$, the set of 
$n \times n$ matrices with complex-valued entries, and by $I_n$ the identity 
matrix in $\bbC^n$.

Let $A= \big[A_{j,k}\big]_{1\leq j,k\leq 2}\in \bbC^{4 \times 4}$, with 
$A_{j,k}\in \bbC^{2 \times 2}$, $1\leq j,k\leq 2$, and $L\in
\bbC^{2 \times 2}$, chosen such that $\ker (A_{1,1}
+A_{1,2} L)=\{0\}$; that is, $(A_{1,1}+A_{1,2} L)$ is invertible in
$\bbC^2$. Define for such $A$ (cf., e.g., \cite{KS74}),
\begin{equation}
M_A(L)=(A_{2,1}+A_{2,2} L)(A_{1,1}+A_{1,2} L)^{-1},   \lb{4.1}
\end{equation}
and observe that
\begin{align}
& M_{I_{4}} (L) = L, \\
& M_{AB}(L) = M_A(M_B(L)), \lb{4.2} \\
& M_{A^{-1}} (M_A(L)) = L = M_A(M_{A^{-1}} (L)), \quad A \, \text{ invertible,}  
\lb{4.2a} \\
& M_A(L)=M_{AB^{-1}}(M_B(L)),   \lb{4.2b}
\end{align}
whenever the right-hand sides (and hence the left-hand sides) in 
\eqref{4.2}--\eqref{4.2b} exist.  

%%%%%%%%%%%%%%%
\begin{theorem}  \lb{t4.1}
Assume that $\te_0, \te_R, \te_0^{\prime},\te_R^{\prime}, \de_0,\de_R,\de_0^\prime,\de_R^\prime\in S_{2 \pi}$, $\de_0^\prime-\de_0\ne 0 \, \text{\rm mod} (\pi)$, 
$\de_R^\prime-\de_R \ne 0 \, \text{\rm mod} (\pi)$, and that 
$z\in\bbC\backslash\big(\sigma(H_{\te_0,\te_R})\cup\sigma(H_{\de_0,\de_R})\big)$.  Then, with $\Ste$ defined as in \eqref{2.14},
\begin{align}
\begin{split}
\Lates(z) &= \big(S_{\de_0^\prime-\de_0,\de_R^\prime-\de_R}\big)^{-1}
\big[S_{\de_0^\prime-\te_0^\prime,\de_R^\prime-\te_R^\prime}+
S_{\te_0^\prime-\de_0,\te_R^\prime-\de_R}\Lades(z) \big]  \label{4.3} \\
& \quad \times
\big[S_{\de_0^\prime-\te_0,\de_R^\prime-\te_R}+
S_{\te_0-\de_0,\te_R-\de_R}\Lades(z) \big]^{-1}S_{\de_0^\prime-\de_0,\de_R^\prime-\de_R}.
\end{split}
\end{align} 
\end{theorem}
%%%%%%%%%%%%%%%
\begin{proof}
Assume $\te_0, \te_R, \te_0^{\prime},\te_R^{\prime}\in S_{2 \pi}$,
$z\in\bbC\backslash\sigma(\Hte)$,
and let $\psi_j(z,\cdot)$, $j=1,2$, denote a basis for the solutions of \eqref{2.3}. Then, with $\Cte$ and $\Ste$ as defined in \eqref{2.14}, equation \eqref{2.13} yields
\begin{align}\label{4.4}
&\Lates (z) \bigg(\Cte\begin{bmatrix} \psi_1(z,0)&\psi_2(z,0)\\ \psi_1(z,R)&\psi_2(z,R)\end{bmatrix}+
\Ste\begin{bmatrix} \psi'_1(z,0)&\psi'_2(z,0)\\ -\psi'_1(z,R)&-\psi'_2(z,R)\end{bmatrix}\bigg)\no  \\
& \quad =\bigg(\Ctes\begin{bmatrix} \psi_1(z,0)&\psi_2(z,0)\\ \psi_1(z,R)&\psi_2(z,R)\end{bmatrix}+
\Stes\begin{bmatrix} \psi'_1(z,0)&\psi'_2(z,0)\\ -\psi'_1(z,R)&-\psi'_2(z,R)\end{bmatrix}\bigg).
\end{align}

From Remark~\ref{r2.6}, recall, for $z\in\bbC\backslash\sigma(H_{0,0})$, that $\Lambda_{D,N}(z)=\Lazzqq(z)=\Lambda_{0,0}(z)$ and note that \eqref{2.32} then yields
\begin{equation}\label{4.5}
\Lambda_{0,0}(z) = \begin{bmatrix} \psi'_1(z,0)&\psi'_2(z,0)\\ -\psi'_1(z,R)&-\psi'_2(z,R)
\end{bmatrix}
\begin{bmatrix} \psi_1(z,0)&\psi_2(z,0)\\ \psi_1(z,R)&\psi_2(z,R)\end{bmatrix}^{-1}.
\end{equation}
Then, with $\Cte$ defined as in \eqref{2.14}, and with
$z\in\bbC\big\backslash\big(\sigma(\Hte)\cup \sigma(H_{0,0})\big)$, \eqref{4.4} can be written as
\begin{equation} \label{4.6}
\Lates(z) =\big[\Ctes+\Stes \Lambda_{0,0}(z)\big]
\big[\Cte+\Ste \Lambda_{0,0}(z)\big]^{-1}.
\end{equation}

Next, assume that $\te_0, \te_R, \te_0^{\prime},\te_R^{\prime}, \de_0,\de_R,\de_0^\prime,\de_R^\prime\in S_{2 \pi}$, and let
$A, B\in \bbC^{4 \times 4}$ be defined by
\begin{equation}\label{4.7}
A= \begin{bmatrix} \Cte & \Ste \\ C_{\te_0^\prime,\te_R^\prime} & S_{\te_0^\prime,\te_R^\prime}\end{bmatrix},\quad
B=\begin{bmatrix} C_{\de_0,\de_R} &
S_{\de_0,\de_R}\\ C_{\de_0^\prime,\de_R^\prime} & S_{\de_0^\prime,\de_R^\prime}\end{bmatrix}.
\end{equation}
Then, by \eqref{4.1}, and \eqref{4.6},
\begin{equation}\lb{4.8}
\Lates(z) = M_A(\Lambda_{0,0}(z)), \qquad \Lades(z) = M_B(\Lambda_{0,0}(z))
\end{equation}
for $z\in\bbC\big\backslash\big(\sigma(\Hte)
\cup\sigma(H_{\de_0,\de_R})\cup \sigma(H_{0,0})\big)$. If additionally, 
one assumes that  
$\de_0^\prime-\de_0\ne 0 \, \text{\rm mod} (\pi)$, 
$\de_R^\prime-\de_R \ne 0 \, \text{\rm mod} (\pi)$, then
\begin{equation}\lb{4.9}
AB^{-1}=
\begin{bmatrix}
\big(S_{\de_0^\prime-\de_0,\de_R^\prime-\de_R}\big)^{-1}S_{\de_0^\prime-\te_0,\de_R^\prime-\te_R} & \big(S_{\de_0^\prime-\de_0,\de_R^\prime-\de_R}\big)^{-1}S_{\te_0-\de_0,\te_R-\de_R}\\ \big(S_{\de_0^\prime-\de_0,\de_R^\prime-\de_R}\big)^{-1}S_{\de_0^\prime-\te_0^\prime,\de_R^\prime-\te_R^\prime} & \big(S_{\de_0^\prime-\de_0,\de_R^\prime-\de_R}\big)^{-1}S_{\te_0^\prime-\de_0,\te_R^\prime-\de_R}
\end{bmatrix},
\end{equation}
and \eqref{4.3} then follows from \eqref{4.1}, \eqref{4.2}, \eqref{4.8} and \eqref{4.9}, given that
\begin{equation}\lb{4.10}
\Lates(z) = M_A(\Lambda_{0,0}(z))=M_{AB^{-1}}(M_B(\Lambda_{0,0}(z))=M_{AB^{-1}}(\Lades(z)).
\end{equation}
By meromorphic continuation, \eqref{4.3} holds for $z\in\bbC\big\backslash\big(\sigma(\Hte)\cup\sigma(H_{\de_0,\de_R})\big)$.
\end{proof}
%%%%%%%%%%%%%%%

If $\Hte$ and $H_{\de_0,\de_R}$ are self-adjoint, then \eqref{4.3} holds for 
$z\in\bbC\backslash\bbR$.  

%%%%%%%%%%%%%%%
\begin{remark} \lb{r4.2}
In the special case of \eqref{4.3} which relates two generalized Dirichlet-to-Neumann maps one obtains 
\begin{align}\label{4.11}
\Late(z)&=\big[-S_{\te_0-\de_0,\te_R-\de_R} + C_{\te_0-\de_0,\te_R-\de_R}
\Lade(z)\big]\notag\\
& \quad \times\big[C_{\te_0-\de_0,\te_R-\de_R}+ S_{\te_0-\de_0,\te_R-\de_R}
\Lade(z)\big]^{-1},    \\
& \hspace*{-.6cm} 
\te_0, \te_R, \de_0,\de_R\in S_{2 \pi}, \;  
z\in\bbC\backslash\big(\sigma(H_{\te_0,\te_R})\cup\sigma(H_{\de_0,\de_R})\big). \no 
\end{align}
\end{remark}
%%%%%%%%%%%%%%%

The following reformulation of \eqref{4.3} (motivated by the form of \eqref{3.33aa} in 
Theorem \ref{t3.5}) will be crucial in the proof of Theorem \ref{t4.6}.  

%%%%%%%%%%%%%%%
\begin{corollary}  \lb{c4.3}
Assume that $\te_0, \te_R, \te_0^{\prime},\te_R^{\prime}, \de_0,\de_R,\de_0^\prime,\de_R^\prime\in S_{2 \pi}$, $\te_0^\prime-\te_0\ne 0 \, \text{\rm mod} (\pi)$, 
$\te_R^\prime-\te_R \ne 0 \, \text{\rm mod} (\pi)$,
$\de_0^\prime-\de_0\ne 0 \, \text{\rm mod} (\pi)$, 
$\de_R^\prime-\de_R \ne 0 \, \text{\rm mod} (\pi)$, and that 
$z\in\bbC\backslash\big(\sigma(H_{\te_0,\te_R})\cup\sigma(H_{\de_0,\de_R})\big)$.  Then  
\begin{align} \lb{4.11a}
& \Lates(z) S_{\te_0'-\te_0,\te_R'-\te_R}   \no \\
& \quad = \Big[S_{\de_0^\prime-\te_0^\prime,\de_R^\prime-\te_R^\prime} + 
\big(S_{\de_0^\prime-\de_0,\de_R^\prime-\de_R}\big)^{-1}
S_{\te_0^\prime-\de_0,\te_R^\prime-\de_R}\Lades(z) 
S_{\de_0^\prime-\de_0,\de_R^\prime-\de_R}\Big]      \no\\
& \qquad \, \times
\Big[\big(S_{\te_0'-\te_0,\te_R'-\te_R}\big)^{-1}
S_{\de_0^\prime-\te_0,\de_R^\prime-\te_R} + 
\big(S_{\te_0'-\te_0,\te_R'-\te_R}\big)^{-1}
\big(S_{\de_0^\prime-\de_0,\de_R^\prime-\de_R}\big)^{-1}   \no \\
& \hspace*{4.5cm} \times S_{\te_0-\de_0,\te_R-\de_R}\Lades(z) 
S_{\de_0'-\de_0,\de_R'-\de_R}\Big]^{-1}.
\end{align}
\end{corollary}
%%%%%%%%%%%%%%%
\begin{proof}
This follows from multiplying \eqref{4.3} by $S_{\te_0'-\te_0,\te_R'-\te_R}$ from 
the right, inserting the terms 
$S_{\de_0^\prime-\de_0,\de_R^\prime-\de_R} 
\big(S_{\de_0^\prime-\de_0,\de_R^\prime-\de_R}\big)^{-1}$ to the right of 
$\Lades(z)$ twice, and then algebraically manipulate the various terms in the 
equation resulting from these insertions into \eqref{4.3} (repeatedly using 
$(ST)^{-1} = T^{-1}S^{-1}$, etc.).  
\end{proof}
%%%%%%%%%%%%%%%

Since the self-adjoint case is the principal focus for the remainder of this section,  
we now recall some additional pertinent facts from \cite{KS74} (see also 
\cite[Sect.\ 6]{GT00}) on linear fractional transformations of matrices. Defining
\begin{equation}\lb{4.12}
J_{4}=\begin{bmatrix} 0 & -I_2\\ I_2 & 0\end{bmatrix},
\end{equation}
and 
\begin{equation}\lb{4.13}
\cA_{4}=\{ A\in\bbC^{4 \times 4} \, | \, A^*J_{4}A=J_{4}\}, 
\end{equation}
representing $A\in \bbC^{4 \times 4}$ by
\begin{equation}\lb{4.14}
A= \begin{bmatrix} A_{1,1} & A_{1,2}\\ A_{2,1} & A_{2,2}\end{bmatrix},
\quad A_{p,q}\in \bbC^{2 \times 2},\quad 1\leq p,q\leq 2,
\end{equation}
the condition $A^*J_{4}A=J_{4}$ in \eqref{4.13} is equivalent to 
\begin{align}
\begin{split}  \lb{4.15}
& A^*_{1,1}A_{2,1}=A^*_{2,1}A_{1,1},\quad
A^*_{2,2}A_{1,2}=A^*_{1,2}A_{2,2},  \\
& A^*_{2,2}A_{1,1}-A^*_{1,2}A_{2,1} = I_2 
= A^*_{1,1}A_{2,2} - A^*_{2,1}A_{1,2},
\end{split} 
\end{align}
or equivalently, to 
\begin{equation}\lb{4.16}
\begin{bmatrix}A^*_{2,2} & -A^*_{1,2}\\-A^*_{2,1}&
A^*_{1,1}\end{bmatrix} \begin{bmatrix} A_{1,1} & A_{1,2}\\
A_{2,1} & A_{2,2}\end{bmatrix} = I_{4}.
\end{equation}
Since left inverses in $\bbC^{4 \times 4}$ are also right inverses, \eqref{4.16}
implies
\begin{equation}\lb{4.17}
\begin{bmatrix} A_{1,1} & A_{1,2}\\ A_{2,1} & A_{2,2}\end{bmatrix}
\begin{bmatrix} A^*_{2,2} & -A^*_{1,2}\\ -A^*_{2,1} &
A^*_{1,1}\end{bmatrix} = I_{4},
\end{equation}
that is,
\begin{align}
\begin{split}  \lb{4.18}
& A_{1,1}A^*_{1,2}=A_{1,2}
A^*_{1,1},\quad A_{2,2}A^*_{2,1}=A_{2,1}A^*_{2,2},   \\
& A_{2,2}A^*_{1,1}-A_{2,1}A^*_{1,2}=I_{2}
=A_{1,1}A^*_{2,2}-A_{1,2} A^*_{2,1},
\end{split}
\end{align}
or equivalently,
\begin{equation}\lb{4.19}
AJ_{4}A^*=J_{4}.
\end{equation}
In particular,
\begin{equation}\lb{4.20}
A\in \cA_{4} \, \text{ if and
only if } \, A^{-1}\in \cA_{4}.
\end{equation}

At this point we turn to the particularly important special self-adjoint case where 
$V$ and $\theta_0, \theta_R, \theta_0', \theta_R'$ are real-valued. In this case 
we will now prove that $\Lates(z) S_{\te_0'-\te_0,\te_R'-\te_R}$ is a $2 \times 2$ 
matrix-valued Herglotz function. But first we note the following result:

%%%%%%%%%%%%%%
\begin{lemma} \lb{l4.4}
Assume that $\te_0, \te_R, \te_0^{\prime},\te_R^{\prime}, \de_0,\de_R,\de_0^\prime,\de_R^\prime\in [0,2 \pi)$, $\te_0^\prime-\te_0\ne 0 \, \text{\rm mod} (\pi)$, 
$\te_R^\prime-\te_R \ne 0 \, \text{\rm mod} (\pi)$,
$\de_0^\prime-\de_0\ne 0 \, \text{\rm mod} (\pi)$, 
$\de_R^\prime-\de_R \ne 0 \, \text{\rm mod} (\pi)$, 
and introduce in accordance with \eqref{4.11a},  
\begin{align}
\begin{split}   \lb{4.22} 
& A(\te, \de) = \big[ A(\te, \de)_{j,k}\big]_{1 \leq j,k \leq 2} \in \bbC^{4 \times 4},   \\
& A(\te, \de)_{1,1} = \big(S_{\te_0'-\te_0,\te_R'-\te_R}\big)^{-1} 
S_{\de_0'-\te_0,\de_R'-\te_R},   \\
& A(\te, \de)_{1,2} = \big(S_{\te_0'-\te_0,\te_R'-\te_R}\big)^{-1} 
\big(S_{\de_0'-\de_0,\de_R'-\de_R}\big)^{-1} S_{\te_0-\de_0,\te_R-\de_R},   \\
& A(\te, \de)_{2,1} = S_{\de_0'-\te_0',\de_R'-\te_R'},    \\ 
& A(\te, \de)_{2,2} = \big(S_{\de_0'-\de_0,\de_R'-\de_R}\big)^{-1} 
S_{\te_0'-\de_0,\te_R'-\de_R}. 
\end{split}
\end{align}
Then
\begin{equation}
A(\te, \de) \in \cA_4.    \lb{4.23} 
\end{equation}
\end{lemma}
%%%%%%%%%%%%%%
\begin{proof}
Since according to \eqref{2.14} $S_{\alpha,\beta}$ are $2 \times 2$ diagonal 
matrices, an entirely straightforward (though, admittedly, rather tedious) explicit computation 
shows that the block matrix entries of $A_{\te, \de}$ in \eqref{4.22} satisfy the 
relations in \eqref{4.15}. 
\end{proof}
%%%%%%%%%%%%%%

We denote by $\bbC_+$  the open complex upper half-plane and abbreviate
$\Im(L)=(L - L^*)/(2i)$ for $L \in\bbC^{n\times n}$, $n\in\bbN$. In addition,
$d \|\Sigma\|_{\bbC^2}$ will denote the total variation of the 
$2\times 2$ matrix-valued measure $d \Sigma$ below in \eqref{5.2}. 

We recall that $M(\cdot)$ is called an $n \times n$ matrix-valued Herglotz function 
if it is analytic on $\bbC_+$ and $\Im(M(z)) \geq 0$ for all $z\in\bbC_+$. In this 
context we also recall the following result:

%%%%%%%%%%%%%%
\begin{lemma} \lb{l4.5}
Assume that $A=\big[A_{j,k}\big]_{1\leq j,k\leq 2}\in \cA_4$ 
and $L \in \bbC^{2 \times 2}$. Then 
\begin{equation}
\Im(L) > 0 \, \text{ implies } \, \ker (A_{1,1} + A_{1,2} L)=\{0\}     \lb{4.24}
\end{equation}
and $M_A(L)=(A_{2,1}+A_{2,2} L)(A_{1,1}+A_{1,2} L)^{-1}$ $($defined according to 
\eqref{4.1}$)$ satisfies
\begin{equation}
\Im(M_A) = \big((A_{2,1} + A_{2,2} L)^{-1}\big)^* \Im(L)    \lb{4.25}
(A_{2,1} + A_{2,2} L)^{-1} > 0. 
\end{equation}
In particular, if $M(\cdot)$ is a $2 \times 2$ matrix-valued Herglotz function satisfying 
\begin{equation}
\Im(M(z)) > 0, \quad z\in\bbC_+,    \lb{4.26} 
\end{equation}
then $M_A(\cdot)$ $($in obvious notation defined according to \eqref{4.1} with $L$ replaced by $M(\cdot)$$)$ is a a $2 \times 2$ matrix-valued Herglotz function satisfying 
\begin{equation}
\Im(M_A(z)) > 0, \quad z\in\bbC_+.    \lb{4.27}
\end{equation}  
\end{lemma}
%%%%%%%%%%%%%%
\begin{proof}
This is the special finite-dimensional case of \cite[Theorem\ 6.4]{GT00}. 
\end{proof}
%%%%%%%%%%%%%%

Now we are in position to prove the fundamental Herglotz property of the matrix 
$\Lates (\cdot) S_{\te_0'-\te_0,\te_R'-\te_R}$ in the case where $\Hte$ is self-adjoint.
 
%%%%%%%%%%%%%%%%%%%%%%%%%%%%%%%%%%%
\begin{theorem} \lb{t4.6}
Let $\te_0, \te_R, \te_0', \te_R' \in [0,2 \pi)$, 
$\te_0^\prime-\te_0\ne 0 \, \text{\rm mod} (\pi)$, 
$\te_R^\prime-\te_R \ne 0 \, \text{\rm mod} (\pi)$, 
$z\in\bbC\backslash\sigma(\Hte)$, and $\Hte$
be defined as in \eqref{2.2a}. In addition, suppose that $V$ is real-valued $($and 
hence $\Hte$ is self-adjoint\,$)$. Then $\Lates (\cdot) S_{\te_0'-\te_0,\te_R'-\te_R}$ 
is a $2\times 2$ matrix-valued Herglotz function admitting the representation
\begin{align}
& \Lates (z) S_{\te_0'-\te_0,\te_R'-\te_R}=\Xi_{\te_0,\te_R}^{\te_0', \te_R'} 
+ \ \int_{{\mathbb{R}}}
d\Sigma_{\te_0,\te_R}^{\te_0', \te_R'} (\lambda)\bigg(\frac{1}{\lambda -z}-\frac{\lambda}
{1+\lambda^2}\bigg),     \lb{5.1} \\
& \hspace*{7.75cm} z\in\bbC\backslash\sigma(\Hte),   \no \\
& \Xi_{\te_0,\te_R}^{\te_0', \te_R'} 
= \Big(\Xi_{\te_0,\te_R}^{\te_0', \te_R'}\Big)^* \in\bbC^{2\times 2},
\quad \int_{\bbR} \f{d\big\|\Sigma_{\te_0,\te_R}^{\te_0', \te_R'} 
(\lambda)\big\|_{\bbC^2}}{1+\lambda^2} <\infty,
\lb{5.2}
\end{align}
where
\begin{align}
\begin{split}
\Sigma_{\te_0,\te_R}^{\te_0', \te_R'} ((\lambda_1,\lambda_2]) =
\f{1}{\pi}\lim_{\delta\downarrow 0}\lim_{\varepsilon\downarrow
0}\int^{\lambda_2+\delta}_{\lambda_1+\delta}d\lambda \,
\Im\Big(\Lates (\lambda +i\varepsilon)\Big),&  \\
 \lambda_1, \lambda_2
\in\bbR, \; \lambda_1<\lambda_2.&  \lb{5.3}
\end{split}
\end{align} 
In addition,
\begin{equation}
\Im\Big(\Lates (z) S_{\te_0'-\te_0,\te_R'-\te_R}\Big) > 0, \quad z\in\bbC_+.    \lb{5.3a}
\end{equation} 
\end{theorem}
%%%%%%%%%%%%%%%%%%%%%%%%%%%%%%%%%%%
\begin{proof}
Without loss of generality, we take $z\in\bbC_+$. An analytic continuation with
respect to $z$ then extends the result \eqref{5.1} to
$z\in\bbC\backslash\sigma(\Hte)$.

We will first prove the Herglotz property for $\Lambda_{\pi/2,\pi/2}$ and then use
a special case of the linear fractional transformation \eqref{4.11a} to conclude that
$\Lates(\cdot) S_{\te_0'-\te_0,\te_R'-\te_R}$ is a $2 \times 2$ matrix-valued Herglotz function for any $\te_0, \te_R, \te_0', \te_R' \in [0,2 \pi)$. Our point of departure is
formula \eqref{3.33a} in the special case $\te_0 = \te_R = \pi/2$, that is,
\begin{equation}
\Lambda_{\pi/2,\pi/2} (z) = \gamma_{0,0}
\big[\gamma_{0,0} (H_{\pi/2,\pi/2}- \ol{z})^{-1}]^*,
\quad z \in \bbC_+,    \lb{5.4}
\end{equation}
noticing that
\begin{equation}
\hatt \gamma_{\pi/2,\pi/2} = \gamma_{\pi,\pi} = - \gamma_{0,0}
\end{equation}
by \eqref{2.2A}.

First, we slightly change the definition of $\gamma_{0,0}$ by introducing 
\begin{equation} \label{5.5}
\wti \gamma_{0,0} \colon \begin{cases}
H^1((0,R)) \rightarrow \bbC^2, \\
u \mapsto \begin{bmatrix} u(0) \\ u(R)  \end{bmatrix}, \end{cases}
\quad  \wti \gamma_{0,0} \in \cB\big(H^1((0,R)), \bbC^2\big),
\end{equation}
instead. One notes that $\wti \gamma_{0,0}$ is well-defined in the following sense: Any
$u \in H^1((0,R))$ has a representative in its equivalence class of Lebesgue  measurable and square integrable elements, again denoted by $u$ for simplicity, 
that is absolutely continuous on $[0,R]$. In fact, by a standard Sobolev embedding result, one has $H^1(0,R))\hookrightarrow C^{1/2}((0,R)) = C^{1/2}([0,R])$. In particular, the limits $\lim_{x\downarrow 0} u(x) = u(0)$ and $\lim_{x\uparrow R} u(x) = u(R)$ are well-defined for this representative $u$.

Next, using
\begin{equation}
H^1((0,R)) \hookrightarrow L^2((0,R); dx)) \hookrightarrow H^1((0,R))^*,    \lb{5.6}
\end{equation}
one infers that (cf.\ also Remark \ref{r3.7})
\begin{equation} \label{5.7}
(\wti \gamma_{0,0})^* \colon \begin{cases}
\bbC^2 \rightarrow H^1((0,R))^*, \\
 \begin{bmatrix} v_1 \\ v_2 \end{bmatrix} \mapsto
v_1 \delta_0 + v_2 \delta_R, \end{cases}  \quad
(\wti \gamma_{0,0})^* \in \cB\big(\bbC^2, H^1((0,R))^*\big),
\end{equation}
where (cf.\ also \cite[Example\ 2]{KS95})
\begin{align}
& \delta_0, \delta_R \in H^1((0,R))^*,   \no \\
& \delta_0 (u)
= {}_{H^1((0,R))}\big\langle \ol u, \delta_0 \big\rangle_{H^1((0,R))^*}
= u(0),    \lb{5.8} \\
& \delta_R (u)
= {}_{H^1((0,R))}\big\langle \ol u, \delta_R \big\rangle_{H^1((0,R))^*}
= u(R), \quad u \in H^1((0,R)),   \no
\end{align}
with ${}_{H^1((0,R))}\langle \, \cdot \, , \cdot \, \rangle_{H^1((0,R))^*}$
denoting the duality
pairing between $H^1((0,R))$ and $H^1((0,R))^*$ (linear in the second 
argument and antilinear in the first, see, e.g., \cite[Sect.\ 2]{GM09a} for 
more details). Indeed, \eqref{5.8}  follows from
\begin{align}
& u(0) v_1 + u(R) v_2 = (\wti \gamma_{0,0} \ol u,v)_{\bbC^2} =
{}_{H^1((0,R))}\big\langle \ol u, (\wti \gamma_{0,0})^* v \big\rangle_{H^1((0,R))^*}, \\
& \hspace*{4.8cm}
v = \begin{bmatrix} v_1 \\ v_2 \end{bmatrix} \in \bbC^2,
\quad u \in H^1((0,R)).   \lb{5.9}
\end{align}

Next, still following the material discussed in \cite[Sect.\ 2]{GM09a}, we extend
the operator $H_{\pi/2,\pi/2}$ in $L^2((0,R); dx)$,
\begin{equation}
(H_{\pi/2,\pi/2} -z I ) \colon \dom\big(H_{\pi/2,\pi/2}\big) \to L^2((0,R); dx),
\quad z \in z\in\bbC_+,     \lb{5.10}
\end{equation}
where $\dom\big(H_{\pi/2,\pi/2}\big) \hookrightarrow L^2((0,R); dx)$, to its extension
$\wti {H_{\pi/2,\pi/2}}$, which maps $H^1((0,R))$ boundedly into $H^1((0,R))^*$,
\begin{equation}
\wti {H_{\pi/2,\pi/2}} \in \cB\big(H^1((0,R)), H^1((0,R))^*\big),   \lb{5.11}
\end{equation}
such that (with $\wti I: H^1((0,R)) \hookrightarrow H^1((0,R))^*$ the continuous embedding operator)
\begin{align}
& \big(\wti {H_{\pi/2,\pi/2}}+\wti I\,\big)\in\cB\big(H^1((0,R)), H^1((0,R))^*\big) \\
& \, \text{and }\big(\wti {H_{\pi/2,\pi/2}} + \wti I\,\big) \colon H^1((0,R)) \to H^1((0,R))^* \,
\text{ is unitary.}   \lb{B.37}
\end{align}
In addition (cf.\ \eqref{2.2f}), 
\begin{align}
& (H_{\pi/2,\pi/2}+I)^{1/2}\in\cB\big(H^1((0,R)), L^2((0,R);dx)\big) \\
& \, \text{and }
(H_{\pi/2,\pi/2} + I)^{1/2} \colon H^1((0,R)) \to L^2((0,R);dx) \, \text{ is unitary}
\lb{B.40}
\end{align}
(cf.\ \cite[Sect.\ 2]{GM09a}). Moreover, $\wti {H_{\pi/2,\pi/2}}$ is self-adjoint,
\begin{equation}
\Big(\wti {H_{\pi/2,\pi/2}}\Big)^* = \wti {H_{\pi/2,\pi/2}},      \lb{5.12}
\end{equation}
in the sense that
\begin{align}
\begin{split}
{}_{H^1((0,R))}\big\langle w_1, \wti {H_{\pi/2,\pi/2}} w_2 \big\rangle_{H^1((0,R))^*}
= \ol{{}_{H^1((0,R))}\big\langle w_2, 
\wti {H_{\pi/2,\pi/2}} w_1 \big\rangle_{H^1((0,R))^*}}&, \\
\quad w_1, w_2 \in H^1((0,R))&    \lb{5.12a}
\end{split}
\end{align}
(again we refer to \cite[Sect.\ 2]{GM09a} for more details).

In addition,
\begin{equation}
\Big(\wti {H_{\pi/2,\pi/2}} -z \wti I \Big)^{-1} \colon H^1((0,R))^* \to H^1((0,R)),
\quad z \in\bbC_+,    \lb{5.13}
\end{equation}
and hence,
\begin{align}
\begin{split}
\Big(\Big(\wti {H_{\pi/2,\pi/2}} -z \wti I \Big)^{-1} w\Big)(x)
 = {}_{H^1((0,R))}\big\langle \ol{G_{\pi/2,\pi/2} (z,x,\cdot)},
w \big\rangle_{H^1((0,R))^*},&    \\
w \in H^1((0,R))^*, \; z \in\bbC_+,&     \lb{5.14}
\end{split}
\end{align}
using the fact that
\begin{equation}
G_{\pi/2,\pi/2} (z,x,\cdot) \in H^1((0,R)), \quad x \in\bbR, \;
z\in\bbC\backslash\sigma(H_{\pi/2,\pi/2}).
\end{equation}

By \eqref{2.5a} and \eqref{2.5b}, the Wronskian $W$ is of the form
\begin{align}
\begin{split}
W(z) & = W(u_{+,\pi/2}(z,\cdot), u_{-,\pi/2}(z,\cdot)) \\
& = - u'_{+,\pi/2}(z,0) u_{-,\pi/2}(z,0) = u_{+,\pi/2}(z,R) u'_{-,\pi/2}(z,R).   \lb{5.15}
\end{split}
\end{align}
Using \eqref{3.32}, one computes
\begin{align}
& \Big(\Big(\wti {H_{\pi/2,\pi/2}} -z \wti I \Big)^{-1} (\wti \gamma_{0,0})^* v \Big)(x)   \no \\
& \quad = \f{1}{W(z)}
{}_{H^1((0,R))}\big\langle \ol{G_{\pi/2,\pi/2} (z,x,\cdot)},
[v_1 \delta_0 + v_2 \delta_R] \big\rangle_{H^1((0,R))^*}    \no \\
& \quad = \f{1}{W(z)} \big(
u_{-,\pi/2}(z,x) u_{+,\pi/2} (z,R) v_2 +
u_{+,\pi/2}(z,x) u_{-,\pi/2} (z,0) v_1 \big),    \lb{5.16} \\
& \hspace*{7.6cm}  v = [v_1 \;\, v_2]^\top \in \bbC^2,  \no
\end{align}
and hence
\begin{align}
& \wti \gamma_{0,0} \Big(\wti {H_{\pi/2,\pi/2}}
- z \wti I \Big)^{-1} (\wti \gamma_{0,0})^* v  \no \\
& \quad = \f{1}{W(z)}\begin{bmatrix} u_{-,\pi/2}(z,0) u_{+,\pi/2}(z,R) v_2
+ u_{-,\pi/2}(z,0) v_1 \\[1mm]
u_{+,\pi/2}(z,R) v_2 + u_{+,\pi/2}(z,R) u_{-,\pi/2}(z,0) v_1
\end{bmatrix}  \no \\[1mm]
& \quad = \f{1}{W(z)} \begin{bmatrix} u_{-,\pi/2}(z,0)
& u_{-,\pi/2}(z,0) u_{+,\pi/2}(z,R)  \\[1mm]
u_{-,\pi/2}(z,0) u_{+,\pi/2}(z,R) & u_{+,\pi/2}(z,R)
\end{bmatrix} \begin{bmatrix} v_1 \\[1mm] v_2 \end{bmatrix}  \no \\[1mm]
& \quad = \Lambda_{\pi/2,\pi/2} (z) v,  \quad
v = [v_1 \;\, v_2]^\top \in \bbC^2,    \lb{5.17}
\end{align}
inserting \eqref{5.15} for $W(\cdot)$. Consequently, one has
\begin{equation}
\Lambda_{\pi/2,\pi/2} (z) =
\wti \gamma_{0,0} \Big(\wti {H_{\pi/2,\pi/2}} - z \wti I \Big)^{-1} (\wti \gamma_{0,0})^*,
\quad z\in\bbC_+,    \lb{5.18}
\end{equation}
and also
\begin{equation}
\Lambda_{\pi/2,\pi/2} (z)^*
= \Lambda_{\pi/2,\pi/2} (\ol z), \quad z\in\bbC_+.    \lb{5.19}
\end{equation}
In particular,
\begin{align}
& \big(v,\Im \big(\Lambda_{\pi/2,\pi/2} (z)\big) v\big)_{\bbC^2}
= \f{1}{2 i} \big(v,
\big(\Lambda_{\pi/2,\pi/2} (z) - \Lambda_{\pi/2,\pi/2} (z)^*\big) v\big)_{\bbC^2}   \no \\
& \quad = \f{1}{2 i} \Big(v,
\wti \gamma_{0,0} \Big(\Big(\wti {H_{\pi/2,\pi/2}} - z \wti I \Big)^{-1}
- \Big(\wti {H_{\pi/2,\pi/2}} -\ol{z} \wti I \Big)^{-1}\Big) (\wti \gamma_{0,0})^*
v\Big)_{\bbC^2}    \no \\
& \quad = \Im(z) \Big(v,
\wti \gamma_{0,0} \Big[\Big(\wti {H_{\pi/2,\pi/2}} - z \wti I \Big)^{-1}
\wti I 
 \Big(\wti {H_{\pi/2,\pi/2}} -\ol{z} \wti I \Big)^{-1}\Big] (\wti \gamma_{0,0})^*
 v\Big)_{\bbC^2}     \no \\
& \quad = \Im(z) \Big(v,
\Big[\wti \gamma_{0,0} \Big(\wti {H_{\pi/2,\pi/2}} - z \wti I \Big)^{-1}\Big]
\wti I \Big[\wti \gamma_{0,0} 
\Big(\wti {H_{\pi/2,\pi/2}} - z \wti I \Big)^{-1}\Big]^* v\Big)_{\bbC^2}    \no \\
& \quad = \Im(z)\, {}_{H^1((0,R))}\Big\langle \Big[\wti \gamma_{0,0} 
\Big(\wti {H_{\pi/2,\pi/2}} - z \wti I \Big)^{-1}\Big]^* v,\wti I    \no \\
& \hspace*{3.2cm}  \times 
\Big[\wti \gamma_{0,0} \Big(\wti {H_{\pi/2,\pi/2}} - z \wti I \Big)^{-1}\Big]^*
v\Big\rangle_{H^1((0,R))^*}   \no \\
& \quad = \Im(z) \big(\big[\gamma_{0,0} ({H_{\pi/2,\pi/2}} - z I)^{-1}\big]^* v, 
\big[\gamma_{0,0} ({H_{\pi/2,\pi/2}} - z I)^{-1}\big]^*v\big)_{L^2((0,R); dx)}   \no \\ 
& \quad = \Im(z) 
\big\|\big[\gamma_{0,0} ({H_{\pi/2,\pi/2}} - z I)^{-1}\big]^*v\big\|_{L^2((0,R); dx)}^2  
\geq 0, \quad  v \in \bbC^2, \; z\in\bbC_+,      \lb{5.20}
\end{align}
since the duality pairing 
${}_{H^1((0,R))}\langle \, \cdot \, , \cdot \, \rangle_{H^1((0,R))^*}$ 
between the spaces $H^1((0,R))$ and $H^1((0,R))^*$ is compatible with the scalar 
product in $L^2((0,R); dx)$, that is,
\begin{equation}
{}_{H^1((0,R))}\langle w_1, \wti I w_2 \rangle_{H^1((0,R))^*} 
= (w_1, w_2)_{L^2((0,R); dx)},  \quad w_1, w_2 \in H^1((0,R)). 
\end{equation}
In particular,
\begin{align}
\begin{split} 
\Im \big(\Lambda_{\pi/2,\pi/2} (z)\big)  
= \big[\gamma_{0,0} ({H_{\pi/2,\pi/2}} - z I)^{-1}\big]
\big[\gamma_{0,0} ({H_{\pi/2,\pi/2}} - z I)^{-1}\big]^* \geq 0,&   \lb{5.20a} \\
z \in \bbC_+,&  
\end{split} 
\end{align}
and thus, $\Lambda_{\pi/2,\pi/2} (\cdot)$ is a $2 \times 2$ matrix-valued Herglotz 
function. In fact, one can improve on \eqref{5.20a} to obtain  
\begin{equation}
\Im \big(\Lambda_{\pi/2,\pi/2} (z)\big)  > 0,  \quad  z\in\bbC_+,  \lb{5.21}
\end{equation}
since 
\begin{equation}
\ker \big(\big[\gamma_{0,0} (H_{\pi/2,\pi/2} - z I)^{-1}\big]^*\big) 
= \{0\},  \quad  z\in\bbC\backslash\sigma(H_{\pi/2,\pi/2}).    \lb{5.22} 
\end{equation}

To prove \eqref{5.22}, one can argue as follows: Suppose that 
\begin{equation} 
[a_0 \;\, a_R]^\top \in 
\ker \big(\big[\gamma_{0,0} (H_{\pi/2,\pi/2} - z I)^{-1}\big]^*\big),  
\end{equation}
then by \eqref{3.43}, 
\begin{align}
& \big(\big[\gamma_{0,0} (H_{\pi/2,\pi/2} - z I)^{-1}\big]^* [a_0 \;\, a_R]^\top\big)(x)
= \f{1}{W(u_{+,\pi/2}(\ol z,\cdot), u_{-,\pi/2}(\ol z,\cdot))}  \no \\
& \qquad \times 
\big[a_0 u_{-,\pi/2}(\ol z,0) u_{+,\pi/2}(\ol z,x) + a_R u_{+,\pi/2}(\ol z,R) 
u_{-,\pi/2}(\ol z,x) \big] =0,    \\
& \hspace*{7.65cm} z \in\bbC\backslash\sigma(H_{\pi/2,\pi/2}).   \no 
\end{align}
Since by definition (cf.\ \eqref{2.5a}, \eqref{2.5b}), $u_{-,\pi/2}' (\ol z,0) =  
u_{+,\pi/2}' (\ol z,R) = 0$, one concludes that 
\begin{equation}
u_{-,\pi/2} (\ol z,0) \neq 0, \quad u_{+,\pi/2} (\ol z,R) \neq 0.  
\end{equation}
Moreover, since $W(u_{+,\pi/2}(\ol z,\cdot), u_{-,\pi/2}(\ol z,\cdot)) \neq 0$ for all  
$z \in\bbC\backslash\sigma(H_{\pi/2,\pi/2})$ (otherwise, $\ol z$ would be an 
eigenvalue of $H_{\pi/2,\pi/2}$), $u_{+,\pi/2}(\ol z,\cdot)$ and 
$u_{-,\pi/2}(\ol z,\cdot))$ are linearly independent, implying 
\begin{equation}
a_0 = a_R = 0,
\end{equation}
and hence \eqref{5.22}.

Next, using the notation introduced in \eqref{3.0}, and applying the linear fractional
transformation \eqref{4.3} one can show that (with $z\in\bbC_+$)
\begin{equation}
\Lambda_{\pi/2,\pi/2} (z) \big[C_{\pi/2,\pi/2} + S_{\pi/2,\pi/2} \Lambda_{0,0} (z)\big]
= \big[C_{\pi,\pi} + S_{\pi,\pi} \Lambda_{0,0} (z)\big],
\end{equation}
or equivalently, that
\begin{equation}
\Lambda_{\pi/2,\pi/2} (z) \Lambda_{0,0} (z) = - I_2,
\end{equation}
and hence,
\begin{equation}
\Lambda_{0,0}^{\pi/2,\pi/2} (z) = \Lambda_{0,0} (z) 
= -\big [\Lambda_{\pi/2,\pi/2} (z)\big]^{-1}
= - \big[\Lambda_{\pi/2,\pi/2}^{\pi,\pi} (z)\big]^{-1},   \lb{4.60}
\end{equation}
is a $2 \times 2$ matrix-valued Herglotz function too, satisfying
\begin{equation}
 \Im \big(\Lambda_{0,0}^{\pi/2,\pi/2} (z)\big) > 0, \quad z\in\bbC_+.    \lb{4.61}
\end{equation} 

The Herglotz property of $\Lates (\cdot) S_{\te_0'-\te_0,\te_R'-\te_R}$, 
$\te_0,\te_R, \te_0', \te_R' \in [0,2 \pi)$ 
then follows again from the linear fractional transformation \eqref{4.11a} and 
from \eqref{4.27} in Lemma \ref{l4.5} upon identifying the $2\times 2$ block matrix
$A = \big[A_{j,k}\big]_{1\leq j, k \leq 2}$ in Lemma \ref{l4.5} with 
$A(\te,\de) \in \cA_4$ in the special case where  
\begin{align}
\begin{split} 
& \te_0, \te_R, \te_0', \te_R' \in [0,2 \pi), \quad 
\te_0^\prime-\te_0\ne 0 \, \text{\rm mod} (\pi),  \quad  
\te_R^\prime-\te_R \ne 0 \, \text{\rm mod} (\pi),     \\ 
& \de_0 = \de_R =0, \quad \de_0' = \de_R' = \pi/2.  \lb{4.62} 
\end{split} 
\end{align}
Given the Herglotz property of $\Lates (\cdot) S_{\te_0'-\te_0,\te_R'-\te_R}$, one 
obtains as in \cite[Theorem\ 5.4]{GT00} the representation
\begin{align}
& \Lates (z) S_{\te_0'-\te_0,\te_R'-\te_R} =\Xi_{\te_0,\te_R}^{\te_0', \te_R'} 
+ \Upsilon_{\te_0,\te_R}^{\te_0', \te_R'} z +\ \int_{{\mathbb{R}}}
d\Sigma_{\te_0,\te_R}^{\te_0', \te_R'} 
(\lambda)\bigg(\frac{1}{\lambda -z}-\frac{\lambda}
{1+\lambda^2}\bigg),     \no \\
& \hspace*{8.4cm} z\in\bbC\backslash\sigma(\Hte),   \lb{5.26} 
\end{align}
with $\Xi_{\te_0,\te_R}^{\te_0', \te_R'}$ and 
$\Sigma_{\te_0,\te_R}^{\te_0', \te_R'} (\cdot)$ as in \eqref{5.2}
and \eqref{5.3}, and with $\Upsilon_{\te_0,\te_R}^{\te_0', \te_R'}$ satisfying
\begin{equation}
0 \leq \Upsilon_{\te_0,\te_R}^{\te_0', \te_R'}\in\bbC^{2\times 2}.    \lb{5.27}
\end{equation}
Thus, to conclude the proof of \eqref{5.1}, it remains to prove that actually, 
$\Upsilon_{\te_0,\te_R}^{\te_0', \te_R'} = 0$. 
The latter fact is clear since
\begin{equation}
\Lates (z) \underset{\substack{|z|\to\infty\\ \Im(z^{1/2})>0}}{=} \Oh(|z|^{1/2}),
\end{equation}
using the fact that by \eqref{4.3}, 
\begin{equation}
\Lates (z) = C_{\te_0'-\te_0,\te_R'-\te_R} 
+ S_{\te_0'-\te_0,\te_R'-\te_R} \Late (z),  
\end{equation}
and applying Lemma \ref{l3.4}.

Finally, \eqref{5.3a} is a consequence of 
\eqref{4.61} and Lemma \ref{l4.5} with $A = A(\te,\de)$ chosen as in \eqref{4.62}. 
\end{proof}
%%%%%%%%%%%%%%%%%%%%%%%%%%%%%%%%%%%

As a particular case of Theorem \ref{t4.6} where 
$\te_0' = (\te_0 + (\pi/2)) \, \text{\rm mod} (2\pi)$, 
$\te_R' = (\te_R + (\pi/2)) \, \text{\rm mod} (2\pi)$, one concludes that 
\begin{equation}
\Late(z) = \Lateq (z), \quad \te_0,\te_R\in [0,2 \pi), \; 
z\in\bbC\backslash \sigma(\Hte),    \lb{4.76}
\end{equation}
is a $2 \times 2$ matrix-valued Herglotz function.

%%%%%%%%%%%%%%%%%%%%%%%%%%%%%%%%%%%%%%%%
%%%%%%%%%%%%%%%%%%%%%%%%%%%%%%%%%%%%%%%%
\section{Krein-Type Resolvent Formulas}  \label{s6}
%%%%%%%%%%%%%%%%%%%%%%%%%%%%%%%%%%%%%%%%
%%%%%%%%%%%%%%%%%%%%%%%%%%%%%%%%%%%%%%%%

Krein-type resolvent formulas have been studied in a great variety of contexts, far too
numerous to account for all in this paper. For instance, they are of fundamental
importance in connection with the spectral and inverse spectral theory of ordinary and
partial differential operators. Abstract versions of Krein-type resolvent formulas (see
also the brief discussion at the end of our introduction), connected to boundary value spaces
(boundary triples) and self-adjoint extensions of closed symmetric operators
with equal (possibly infinite) deficiency spaces, have received enormous
attention in the literature. In particular, we note that Robin-to-Dirichlet
maps in the context of ordinary differential operators reduce to the
celebrated (possibly, matrix-valued) Weyl--Titchmarsh function, the basic
object of spectral analysis in this context.  Since it is impossible to
cover the literature in this paper, we refer the reader to the rather extensive
recent bibliography in \cite{GM08}, \cite{GM09}, and \cite{GM10}. Here we 
just mention, for instance,
\cite[Sect.\ 84]{AG81}, \cite{ABMN05}, \cite{AB09}, \cite{AP04}, \cite{AT03}, 
\cite{AT05},
\cite{BL07}, \cite{BMN08}, \cite{BMN02}, \cite{BGW09},
\cite{BHMNW09}, \cite{BM04}, \cite{BMNW08}, \cite{BGP08},
\cite{DHMS06}, \cite{DM91}, \cite{DM95}, \cite{GKMT01},
\cite{GMT98}, \cite{GMZ07}, \cite{GT00}, \cite[Ch.\ 3]{GG91},
\cite{Gr08a}, \cite[Ch.\ 13]{Gr09},\cite{KO77}, \cite{KO78}, \cite{KS66},
\cite{Ku09}, \cite{KK04}, \cite{LT77}, \cite{MM06},  \cite{Ma04},
\cite{Ne83}, \cite{Pa06}, \cite{Pa87}, \cite{Pa02}, \cite{Pa08}, \cite{Po01}, \cite{Po03}, 
\cite{Po04}, \cite{Po08},
\cite{PR09}, \cite{Ry07}, \cite{Ry09}, \cite{Ry10}, \cite{Sa65}, \cite{St70a},
\cite{TS77}, and the references cited therein.

We start by explicitly computing operators of the type
$\gamma_{\te_0^{\prime},\te_R^{\prime}}(\Hte - z I)^{-1}$ which play a role
at various places in this manuscript (cf.\ Theorem \ref{t3.5}, Lemma \ref{l6.2}, and
Theorem \ref{t6.3}):

%%%%%%%%%%%%%%%
\begin{lemma} \lb{l6.1}
Assume that $\te_0, \te_R, \te_0', \te_R' \in S_{2 \pi}$, let $\Hte$ be defined as in
\eqref{2.2a}, and suppose that $z\in\bbC\backslash\sigma(\Hte)$. Then,
assuming $f \in L^2((0,R); dx)$, and writing
\begin{align}
\gamma_{\te_0^{\prime},\te_R^{\prime}}(\Hte - z I)^{-1} f
= \begin{bmatrix}
\big(\gamma_{\te_0^{\prime},\te_R^{\prime}}(\Hte - z I)^{-1}\big)_1f \\[2mm]
\big(\gamma_{\te_0^{\prime},\te_R^{\prime}}(\Hte - z I)^{-1}\big)_2 f
\end{bmatrix} \in \bbC^2,    \lb{6.1}
\end{align}
one has
\begin{align}
\big(\gamma_{\te_0^{\prime},\te_R^{\prime}}(\Hte - z I)^{-1}\big)_1f
&= \f{\sin(\te_0'-\te_0)}{W(u_{+,\te_R}(z,\cdot), u_{-,\te_0}(z,\cdot))}
\big(\ol{u_{+,\te_R}(z,\cdot)}, f)_{L^2((0,R); dx)}   \no \\
& \quad \times \begin{cases} - \f{u_{-,\te_0}(z,0)}{\sin(\te_0)},
& \te_0 \in S_{2 \pi}\backslash\{0,\pi\}, \\
\f{u'_{-,\te_0}(z,0)}{\cos(\te_0)}, & \te_0 \in S_{2 \pi}\backslash\{\pi/2,3\pi/2\},
 \end{cases}   \lb{6.2} \\
\big(\gamma_{\te_0^{\prime},\te_R^{\prime}}(\Hte - z I)^{-1}\big)_2 f
&= \f{\sin(\te_R'-\te_R)}{W(u_{+,\te_R}(z,\cdot), u_{-,\te_0}(z,\cdot))}
\big(\ol{u_{-,\te_0}(z,\cdot)}, f)_{L^2((0,R); dx)}   \no \\
& \quad \times \begin{cases} - \f{u_{+,\te_R}(z,R)}{\sin(\te_R)},
& \te_R \in S_{2 \pi}\backslash\{0,\pi\}, \\
\f{u'_{+,\te_R}(z,R)}{\cos(\te_R)}, & \te_R \in S_{2 \pi}\backslash\{\pi/2,3\pi/2\},
 \end{cases}    \lb{6.3}
\end{align}
in particular,
\begin{align}
& |\big(\gamma_{\te_0^{\prime},\te_R^{\prime}}(\Hte - z I)^{-1}\big)_1f|
\underset{\te_0'\to \te_0}{=} \Oh(\te_0' - \te_0) C_1(z) \|f\|_{L^2((0,R); dx)},
\lb{6.4} \\
& |\big(\gamma_{\te_0^{\prime},\te_R^{\prime}}(\Hte - z I)^{-1}\big)_2f|
\underset{\te_R'\to \te_R}{=} \Oh(\te_R' - \te_R) C_2(z) \|f\|_{L^2((0,R); dx)},
\lb{6.5}
\end{align}
for some constants $C_j(z) > 0$, $j=1,2$, and hence
\begin{align}
& \gamma_{\te_0,\te_R}(\Hte - z I)^{-1} = 0 \,
\text{ in $\cB\big(L^2((0,R); dx), \bbC^2\big)$},    \lb{6.5a} \\
& \big(\gamma_{\te_0,\te_R}(\Hte - z I)^{-1}\big)_k = 0  \,
\text{ in $\cB\big(L^2((0,R); dx), \bbC\big)$}, \; k=1,2.  \lb{6.5b}
\end{align}
\end{lemma}
%%%%%%%%%%%%%%%
\begin{proof}
Employing \eqref{3.32} and \eqref{3.33} one obtains
\begin{align}
\gamma_{\te_0^{\prime},\te_R^{\prime}}(\Hte - z I)^{-1} f
& = \gamma_{\te_0^{\prime},\te_R^{\prime}} \bigg(\int_0^R dx' \,
G_{\te_0,\te_R}(z,\cdot,x') f(x')\bigg)   \no \\
& = \begin{bmatrix} \bigg(\gamma_{\te_0^{\prime},\te_R^{\prime}} \bigg(\int_0^R dx' \,
G_{\te_0,\te_R}(z,\cdot,x') f(x')\bigg)\bigg)_1 \\[4mm]
\bigg(\gamma_{\te_0^{\prime},\te_R^{\prime}} \bigg(\int_0^R dx' \,
G_{\te_0,\te_R}(z,\cdot,x') f(x')\bigg)\bigg)_2 \end{bmatrix}   \lb{6.6}
\end{align}
and hence (with $W(z) = W(u_{+,\te_R}(z,\cdot), u_{-,\te_0}(z,\cdot))$)
\begin{align}
& \bigg(\gamma_{\te_0^{\prime},\te_R^{\prime}} \bigg(\int_0^R dx' \,
G_{\te_0,\te_R}(z,\cdot,x') f(x')\bigg)\bigg)_1   \no \\
& \quad = \f{1}{W(z)} \bigg\{\cos(\te_0') \bigg(\int_0^R dx' \,
G_{\te_0,\te_R}(z,0,x') f(x')\bigg)   \no \\
& \hspace*{1.95cm} + \sin(\te_0') \bigg(\int_0^R dx' \,
\bigg(\f{\partial}{\partial x} G_{\te_0,\te_R}(z,x,x')\bigg|_{x<x'}\bigg)\bigg|_{x=0}
 f(x')\bigg)\bigg\}  \no \\
& \quad = \f{1}{W(z)} \bigg\{\big[\cos(\te_0') u_{-,\te_0}(z,0)
+ \sin(\te_0') u_{+,\te_0}'(z,0)\big]
\int_0^R dx' \, u_{+,\te_R}(z,x') f(x')\bigg\}  \no \\
& \quad = \f{1}{W(z)} \big(\ol{u_{+,\te_R}(z,\cdot)}, f)_{L^2((0,R); dx)}  \no \\
& \qquad \, \times
\begin{cases} [\cos(\te_0') - \sin(\te_0') \cot(\te_0)] u_{-,\te_0}(z,0),
& \te_0 \in S_{2 \pi}\backslash\{0,\pi\}, \\
 [- \cos(\te_0') \tan(\te_0) + \sin(\te_0')] u_{-,\te_0}'(z,0),
& \te_0 \in S_{2 \pi}\backslash\{\pi/2,3\pi/2\}.
\end{cases}     \lb{6.7}
\end{align}
Equation \eqref{6.7} is easily seen to be equivalent to \eqref{6.2}. Equation \eqref{6.3}
is derived analogously.
\end{proof}
%%%%%%%%%%%%%%%

Introducing the orthogonal projections in $\bbC^2$,
\begin{equation}
P_1 = \begin{bmatrix} 1 & 0 \\ 0 & 0 \end{bmatrix}, \quad
P_2  = \begin{bmatrix} 0 & 0 \\ 0 & 1 \end{bmatrix},    \lb{6.8}
\end{equation}
one obtains the following result, patterned after \cite[Lemma\ 6]{Na01} in the
context of Schr\"odinger operators with Dirichlet and Neumann boundary conditions
on a cube in $\bbR^n$:

%%%%%%%%%%%%%%%
\begin{lemma} \lb{l6.2}
Assume that $\te_0, \te_R, \te_0', \te_R' \in S_{2 \pi}$, let $\Hte$ and
$H_{\theta_0',\theta_R'}$ be defined as in \eqref{2.2a}, and suppose that
$z\in\bbC\big\backslash\big(\sigma(\Hte)\cup \sigma(H_{\theta_0',\theta_R'})\big)$.
Then
\begin{align}
& (H_{\theta_0',\theta_R'} - z I)^{-1} = (\Hte - z I)^{-1}  \no \\
& \quad + \big[\gamma_{\ol{\te_0'},\ol{\te_R'}} ((\Hte)^* - {\ol z} I)^{-1}\big]^*
S_{\te_0'-\te_0, \te_R'-\te_R}^{-1}
\big[\gamma_{\te_0,\te_R} (H_{\theta_0',\theta_R'} - z I)^{-1}\big],   \lb{6.9} \\
& \hspace*{7.95cm} \te_0' \neq \te_0, \; \te_R' \neq \te_R,   \no \\
& (H_{\theta_0,\theta_R'} - z I)^{-1} = (\Hte - z I)^{-1}  \no \\
& \quad + \big[\gamma_{\ol{\te_0},\ol{\te_R'}} ((\Hte)^* - {\ol z} I)^{-1}\big]^*
[\sin(\te_R' - \te_R)]^{-1}P_2 
\big[\gamma_{\te_0,\te_R} (H_{\theta_0,\theta_R'} - z I)^{-1}\big],    \no \\
& \hspace*{9.5cm}  \te_R' \neq \te_R,   \lb{6.10} \\
& (H_{\theta_0',\theta_R} - z I)^{-1} = (\Hte - z I)^{-1}  \no \\
& \quad + \big[\gamma_{\ol{\te_0'},\ol{\te_R}} ((\Hte)^* - {\ol z} I)^{-1}\big]^*
[\sin(\te_0' - \te_0)]^{-1}P_1
\big[\gamma_{\te_0,\te_R} (H_{\theta_0',\theta_R} - z I)^{-1}\big],    \no \\
& \hspace*{9.6cm} \te_0' \neq \te_0.     \lb{6.11}
\end{align}
\end{lemma}
%%%%%%%%%%%%%%%
\begin{proof}
We first consider the case
\begin{equation}
\te_0, \te_R, \te_0', \te_R' \in S_{2 \pi}\backslash\{0,\pi\}, \quad
\te_0' \neq \te_0, \quad \te_R' \neq \te_R,    \lb{6.12}
\end{equation}
which illustrates the principal idea of the proof. To get started, we pick 
$f, g \in L^2((0,R); dx)$ and introduce
\begin{align}
\begin{split}
\phi &= ((\Hte)^* - \ol z I)^{-1} f \in \dom((\Hte)^*),    \\
\psi &= (H_{\te_0',\te_R'} - z I)^{-1} g \in \dom(H_{\te_0',\te_R'}).   \lb{6.13}
\end{split}
\end{align}
Then one computes
\begin{align}
& ((\Hte)^* - \ol z I) \phi, \psi)_{L^2((0,R); dx)} -
(\phi, (H_{\te_0',\te_R'} - z I) \psi)_{L^2((0,R); dx)}   \no \\
& \quad = - \int_0^R dx \, \ol{\phi''(x)} \psi(x) + \int_0^R dx \, \ol{\phi(x)} \psi''(x) \no \\
& \quad = - \ol{\phi'(R)} \psi(R) + \ol{\phi'(0)} \psi(0) +
\ol{\phi(R)} \psi'(R) - \ol{\phi(0)} \psi'(0)    \no \\
& \quad = [\cot(\te_R')-\cot(\te_R)] \ol{\phi(R)} \psi(R) +
[\cot(\te_0')-\cot(\te_0)] \ol{\phi(0)} \psi(0)    \no \\
& \quad = - \f{\sin(\te_R'-\te_R)}{\sin(\te_R')\sin(\te_R)}  \ol{\phi(R)} \psi(R)
- \f{\sin(\te_0'-\te_0)}{\sin(\te_0')\sin(\te_0)} \ol{\phi(0)} \psi(0),      \lb{6.14}
\end{align}
using the fact that \eqref{6.13} implies 
\begin{align}
\begin{split}
\cos(\ol \te_0) \phi(0) + \sin(\ol \te_0) \phi'(0) &= 0,  \\
\cos(\ol \te_R) \phi(R) - \sin(\ol \te_R) \phi'(R) &= 0,  \\
\cos(\te_0') \psi(0) + \sin(\te_0') \psi'(0) &= 0,      \lb{6.15}  \\
\cos(\te_R') \psi(R) - \sin(\te_R') \psi'(R) &= 0.
\end{split}
\end{align}

Using \eqref{6.15} once again, one also computes
\begin{align}
(\gamma_{\ol{\te_0'},\ol{\te_R'}} \phi, \gamma_{\te_0,\te_R} \psi)_{\bbC^2} =
- \f{\sin^2(\te_R'-\te_R)}{\sin(\te_R')\sin(\te_R)}  \ol{\phi(R)} \psi(R)
- \f{\sin^2(\te_0'-\te_0)}{\sin(\te_0')\sin(\te_0)} \ol{\phi(0)} \psi(0).      \lb{6.16}
\end{align}

A comparison of \eqref{6.14} and \eqref{6.16} then yields
\begin{align}
& ((\Hte)^* - \ol z I) \phi, \psi)_{L^2((0,R); dx)} -
(\phi, (H_{\te_0',\te_R'} - z I) \psi)_{L^2((0,R); dx)}   \no \\
& \quad = \big(\gamma_{\ol{\te_0'},\ol{\te_R'}} \phi,
S_{\te_0'-\te_0,\te_R'-\te_R}^{-1}
\gamma_{\te_0,\te_R} \psi\big)_{\bbC^2},    \lb{6.17}
\end{align}
or equivalently,
\begin{align}
& \big(f,(H_{\theta_0',\theta_R'} - z I)^{-1} g\big)_{L^2((0,R); dx)}
= \big(f, (\Hte - z I)^{-1} g\big)_{L^2((0,R); dx)}   \no \\
& \quad + \big(f,
\big[\gamma_{\ol{\te_0'},\ol{\te_R'}} ((\Hte)^* - {\ol z} I)^{-1}\big]^*
S_{\te_0'-\te_0, \te_R'-\te_R}^{-1}     \lb{6.18} \\
& \qquad \times \big[\gamma_{\te_0,\te_R} (H_{\theta_0',\theta_R'} - z I)^{-1}\big]
g)_{L^2((0,R); dx)},    \no
\end{align}
and hence \eqref{6.9} since $f, g \in L^2((0,R); dx)$ are arbitrary, under the
additional assumptions in \eqref{6.12}.

Employing Lemma \ref{l6.1} (and particularly, \eqref{6.4}--\eqref{6.5b}) then
shows that \eqref{6.9} is continuous in $\te_0, \te_R, \te_0', \te_R' \in S_{2 \pi}$
with respect to the norm in $\cB\big(L^2((0,R); dx)\big)$, removing the restrictions
$\te_0, \te_R, \te_0', \te_R' \in S_{2 \pi}\backslash\{0,\pi\}$ in \eqref{6.18}, and thus proving \eqref{6.9}.

Analogous considerations imply \eqref{6.10} and \eqref{6.11}.
\end{proof}
%%%%%%%%%%%%%%%

The principal result of this section, Krein's formula for the difference of resolvents of
$H_{\theta_0',\theta_R'}$ and $\Hte$, then reads as follows:

%%%%%%%%%%%%%%%%%%%%%%%%%%%%%%%%%%%
\begin{theorem} \lb{t6.3}
Assume that $\te_0, \te_R, \te_0', \te_R' \in S_{2 \pi}$, let $\Hte$ and
$H_{\theta_0',\theta_R'}$ be defined as in \eqref{2.2a}, and suppose that
$z\in\bbC\big\backslash\big(\sigma(\Hte)\cup \sigma(H_{\theta_0',\theta_R'})\big)$.
Then, with $\Lates (z)$ introduced in \eqref{2.6},
\begin{align}
& (H_{\theta_0',\theta_R'} -z I)^{-1} = (\Hte -z I)^{-1}  \no \\
& \quad - \big[\gamma_{\ol{\theta_0^{\prime}},\ol{\theta_R^{\prime}}}
((\Hte)^* - {\ol z} I)^{-1}\big]^* S_{\te'_0-\te'_0,\te'_R-\te_R}^{-1}   \lb{6.21} \\
 & \qquad \times \Big[\Lates (z)\Big]^{-1}
 \big[\gamma_{\theta_0^{\prime},\theta_R^{\prime}}(\Hte - z I)^{-1}\big],
\quad \te_0 \neq \te_0', \; \te_R \neq \te_R'.  \no \\
& (H_{\theta_0,\theta_R'} -z I)^{-1} = (\Hte -z I)^{-1}  \no \\
& \quad - \big[\gamma_{\ol{\theta_0},\ol{\theta_R^{\prime}}}
((\Hte)^* - {\ol z} I)^{-1}\big]^* [\sin(\te'_R-\te_R)]^{-1} P_2   \lb{6.22} \\
 & \qquad \times \Big[\Lambda_{\te_0,\te_R}^{\te_0,\te_R'} (z)\Big]^{-1} P_2 
 \big[\gamma_{\theta_0,\theta_R^{\prime}}(\Hte - z I)^{-1}\big],
\quad \te_R \neq \te_R',  \no \\
& (H_{\theta_0',\theta_R} -z I)^{-1} = (\Hte -z I)^{-1}  \no \\
& \quad - \big[\gamma_{\ol{\theta_0'},\ol{\theta_R}}
((\Hte)^* - {\ol z} I)^{-1}\big]^* [\sin(\te_0'-\te_0)]^{-1} P_1  \lb{6.23} \\
 & \qquad \times \Big[\Lambda_{\te_0,\te_R}^{\te_0',\te_R} (z)\Big]^{-1} P_1 
 \big[\gamma_{\theta_0',\theta_R}(\Hte - z I)^{-1}\big],
\quad \te_0 \neq \te_0'.  \no
\end{align}
\end{theorem}
%%%%%%%%%%%%%%%%%%%%%%%%%%%%%%%%%%%
\begin{proof}
Taking adjoints on both sides of \eqref{6.9}, and subsequently replacing $z$ 
by $\ol z$, $\te_0, \te_R$ by $\ol \te_0, \ol \te_R$, and $V$ by $\ol V$, then yields
\begin{align}
& (H_{\theta_0',\theta_R'} - z I)^{-1} = (\Hte - z I)^{-1}  \no \\
& \quad +
\big[\gamma_{\ol{\te_0},\ol{\te_R}} ((H_{\theta_0',\theta_R'})^* - {\ol z} I)^{-1}\big]^*
S_{\te_0'-\te_0, \te_R'-\te_R}^{-1}
\big[\gamma_{\te_0',\te_R'} (\Hte - z I)^{-1}\big],   \lb{6.24} \\
& \hspace*{7.95cm} \te_0' \neq \te_0, \; \te_R' \neq \te_R.   \no
\end{align}
Applying $\gamma_{\te_0,\te_R}$ to both sides of \eqref{6.24}, and using the
fact that $\gamma_{\te_0,\te_R}  (\Hte - z I)^{-1} = 0$ by \eqref{6.5a}, one
obtains
\begin{align}
& \gamma_{\te_0,\te_R} (H_{\theta_0',\theta_R'} - z I)^{-1} =
\gamma_{\te_0,\te_R}
\big[\gamma_{\ol{\te_0},\ol{\te_R}} ((H_{\theta_0',\theta_R'})^* - {\ol z} I)^{-1}\big]^*
\no \\
& \quad \times S_{\te_0'-\te_0, \te_R'-\te_R}^{-1}
\big[\gamma_{\te_0',\te_R'} (\Hte - z I)^{-1}\big],  \quad
\te_0' \neq \te_0, \; \te_R' \neq \te_R.    \lb{6.25}
\end{align}
An insertion of \eqref{6.25} into \eqref{6.9} implies
\begin{align}
& (H_{\theta_0',\theta_R'} - z I)^{-1} = (\Hte - z I)^{-1}  \no \\
& \quad + \big[\gamma_{\ol{\te_0'},\ol{\te_R'}} ((\Hte)^* - {\ol z} I)^{-1}\big]^*
S_{\te_0'-\te_0, \te_R'-\te_R}^{-1}
\gamma_{\te_0,\te_R}
\big[\gamma_{\ol{\te_0},\ol{\te_R}} ((H_{\theta_0',\theta_R'})^* - {\ol z} I)^{-1}\big]^*
\no \\
& \qquad \times S_{\te_0'-\te_0, \te_R'-\te_R}^{-1}
\big[\gamma_{\te_0',\te_R'} (\Hte - z I)^{-1}\big],  \quad
\te_0' \neq \te_0, \; \te_R' \neq \te_R.   \lb{6.26}
\end{align}
Using \eqref{2.48} and \eqref{3.33aa} one obtains
\begin{equation}
\gamma_{\te_0,\te_R} \big[\gamma_{\ol{\te_0},\ol{\te_R}}
((H_{\te_0',\te_R'})^*- {\ol z} I)^{-1}\big]^* = -
\Big[\Lates (z)\Big]^{-1} S_{\te_0'-\te_0,\te_R'-\te_R},
\quad z \in \bbC_+,  \lb{6.27}
\end{equation}
and inserting \eqref{6.27} into \eqref{6.26} yields \eqref{6.21}.

Since equations \eqref{6.22} and \eqref{6.23} are proved similarly, we just briefly 
sketch the proof of \eqref{6.22}: First, in analogy to \eqref{6.24}, one derives from 
\eqref{6.10} that 
\begin{align}
& (H_{\theta_0,\theta_R'} - z I)^{-1} = (\Hte - z I)^{-1}  \no \\
& \quad + \big[\gamma_{\ol{\te_0},\ol{\te_R}} ((H_{\te_0,\te_R'})^* - {\ol z} I)^{-1}\big]^*
[\sin(\te_R' - \te_R)]^{-1}P_2 
\big[\gamma_{\te_0,\te_R'} (H_{\theta_0,\theta_R} - z I)^{-1}\big],    \no \\
& \hspace*{9.5cm}  \te_R' \neq \te_R.   \lb{6.28}
\end{align}
Applying $\gamma_{\te_0,\te_R}$ to both sides of \eqref{6.28}, and using again the
fact that $\gamma_{\te_0,\te_R}  (\Hte - z I)^{-1} = 0$, one obtains
\begin{align}
& \gamma_{\te_0,\te_R} (H_{\theta_0,\theta_R'} - z I)^{-1} =
\gamma_{\te_0,\te_R}
\big[\gamma_{\ol{\te_0},\ol{\te_R}} ((H_{\theta_0,\theta_R'})^* - {\ol z} I)^{-1}\big]^*
\no \\
& \quad \times [\sin(\te_R'-\te_R)]^{-1} P_2 
\big[\gamma_{\te_0,\te_R'} (\Hte - z I)^{-1}\big],  \quad 
\te_R' \neq \te_R.    \lb{6.29}
\end{align}
Finally, an insertion of \eqref{6.29} into the right-hand side of \eqref{6.10}, and using 
\eqref{6.27} in the special case $\te_0'=\te_0$, that is,
\begin{align}
& \gamma_{\te_0,\te_R} \big[\gamma_{\ol{\te_0},\ol{\te_R}}
((H_{\te_0,\te_R'})^*- {\ol z} I)^{-1}\big]^* 
= - \Big[\Lambda_{\te_0,\te_R}^{\te_0,\te_R'} (z)\Big]^{-1} S_{0,\te_R'-\te_R}  \no \\
& \quad = - \Big[\Lambda_{\te_0,\te_R}^{\te_0,\te_R'} (z)\Big]^{-1} 
[\sin(\te_R'-\te_R)]^{-1} P_2, \quad z \in \bbC_+,   
\end{align}
yields \eqref{6.22}.
\end{proof}
%%%%%%%%%%%%%%%%%%%%%%%%%%%%%%%%%%%

%%%%%%%%%%%%%%%%%%%%%%%%%%%%%%%%%%%%%%%%
%%%%%%%%%%%%%%%%%%%%%%%%%%%%%%%%%%%%%%%%
\section{A Brief Outlook} 
\label{s7}
%%%%%%%%%%%%%%%%%%%%%%%%%%%%%%%%%%%%%%%%
%%%%%%%%%%%%%%%%%%%%%%%%%%%%%%%%%%%%%%%%

In this section we provide a brief comparison between $\Late (z)$ and the 
$2 \times 2$ matrix-valued Weyl--Titchmarsh function associated with $\Hte$ 
in the self-adjoint context, 
that is, under the assumptions $\te_0, \te_R \in [0, 2\pi)$ and $V$ is real-valued 
in addition to \eqref{2.2aa}. While both $2 \times 2$ matrices are matrix-valued 
Herglotz functions, they are quite different as this brief section will show. A more 
detailed discussion of the interrelations between these two matrices is beyond the 
scope of this paper and will be taken up elsewhere.

To exhibit some of the differences between $\Late (z)$ and various versions of the 
Weyl--Titchmarsh matrix we will link the entries in both matrices to the Green's 
function $G_{\te_0,\te_R} (z,x,x')$ of $\Hte$ and provide some formulas which may well be of independent interest.

We start by linking $\Late (z)$ and $G_{\te_0,\te_R} (z,x,x')$ and list a variety of pertinent formulas (choosing $z\in\bbC\backslash\bbR$ for notational simplicity):
\begin{align}
\Late (z)_{1,1} &= m_{+,\te_0}(z,\te_R) \no \\
&= \dfrac{-\sin(\te_0) + \cos(\te_0)m_{+,0}(z,\te_R)}
{\cos(\te_0) + \sin(\te_0)m_{+,0}(z,\te_R)}  \no \\[1mm]
&= \dfrac{-\sin(\te_0) + \cos(\te_0)u_{+,\te_R}'(z,0)}{\cos(\te_0) 
+ \sin(\te_0)u_{+,\te_R}'(z,0)}    \no \\[1mm] 
&= \f{1}{\sin^2(\te_0)} \big[G_{\te_0,\te_R} (z,0,0) + \sin(\te_0)\cos(\te_0)\big],   
\lb{7.1}\\
&\hspace*{1.95cm} \te_0 \in [0, 2\pi)\backslash\{0,\pi\}, \, \te_R \in [0, 2\pi),   \no \\
\Lambda_{0,0} (z)_{1,1} & = m_{+,0}(z,0) \no \\
&= \lim_{0<x<x', \, x'\downarrow 0} \partial_x \partial_{x'}
G_{0,0} (z,x,x').    \lb{7.2}
\end{align}
Here we used \eqref{3.24} and (cf.\ also \eqref{3.11b})
\begin{align}
\begin{split}
G_{\te_0,\te_R} (z,0,0) &= \f{- \sin(\te_0)}{\cos(\te_0) + \sin(\te_0) m_{+,0}(z,\te_R)} 
 \\[1mm] 
&= \sin(\te_0) \big[- \cos(\te_0) + \sin(\te_0) m_{+,\te_0}(z,\te_R)\big].   \lb{7.3}
\end{split} 
\end{align}
In the same manner one obtains
\begin{align}
\Late (z)_{2,2} &= - m_{-,\te_R}(z,\te_0) \no \\
&= - \dfrac{\sin(\te_R) + \cos(\te_R)m_{-,0}(z,\te_0)}
{\cos(\te_R) - \sin(\te_R)m_{-,0}(z,\te_0)}   \no \\[1mm]
&= - \dfrac{\sin(\te_R) + \cos(\te_R)u_{-,\te_0}'(z,R)}{\cos(\te_R)
- \sin(\te_R)u_{-,\te_0}'(z,R)}   \no \\[1mm] 
&= \f{1}{\sin^2(\te_R)} \big[G_{\te_0,\te_R} (z,R,R) + \sin(\te_R)\cos(\te_R)\big], 
\lb{7.4}\\
&\hspace*{2.35cm} \te_0 \in [0, 2\pi), \, \te_R \in [0, 2\pi)\backslash\{0,\pi\},   \no \\
\Lambda_{0,0} (z)_{2,2} &= - m_{-,0}(z,0) \no \\
&= \lim_{0<x<x', \, x\uparrow R} \partial_x \partial_{x'}
G_{0,0} (z,x,x').    \lb{7.5}
\end{align}
Here we used \eqref{3.28} and (cf.\ also \eqref{3.25})
\begin{align}
\begin{split}
G_{\te_0,\te_R} (z,R,R) &= \f{- \sin(\te_R)}{\cos(\te_R) - \sin(\te_R) m_{-,0}(z,\te_0)} 
 \\[1mm] 
&= \sin(\te_R) \big[- \cos(\te_R) - \sin(\te_R) m_{-,\te_R}(z,\te_0)\big].   \lb{7.6}
\end{split}
\end{align}
Similarly, the off-diagonal terms of $\Late (z)$ in \eqref{3.11a} can be written as
\begin{align}
& \Late (z)_{1,2} = \Late (z)_{2,1}   \no \\
& \quad = \f{1}{\sin(\te_R)} G_{\te_0,\te_R}  (z,R,R) 
\begin{cases}
\f{- u_{-,\te_0}'(z,0)}{\cos(\te_0)}, & \te_0 \in [0, 2\pi)\backslash\{\pi/2, 3\pi/2\}, \\
\f{u_{-,\te_0}(z,0)}{\sin(\te_0)}, & \te_0 \in [0, 2\pi)\backslash\{0, \pi\},\end{cases}    
\lb{7.7} \\ 
& \hspace{6.65cm}  \te_R \in [0, 2\pi)\backslash\{0,\pi\},   \no \\ 
& \quad = \f{1}{\sin(\te_0)} G_{\te_0,\te_R}  (z,0,0) 
\begin{cases}
\f{u_{+,\te_R}'(z,R)}{\cos(\te_R)}, & \te_R \in [0, 2\pi)\backslash\{\pi/2, 3\pi/2\}, \\
\f{u_{+,\te_R}(z,R)}{\sin(\te_R)}, & \te_R \in [0, 2\pi)\backslash\{0, \pi\},\end{cases}    
\lb{7.8} \\ 
& \hspace{6.45cm}  \te_0 \in [0, 2\pi)\backslash\{0,\pi\}.   \no
\end{align}

Next we turn to (variants of) the $2 \times 2$ Weyl--Titchmarsh matrix 
corresponding to $\Hte$ with respect to an interior reference point 
$x_0 \in (0,R)$. We start by introducing 
(again, choosing $z\in\bbC\backslash\bbR$ for notational simplicity) 
\begin{align}
&m_{+,0} (z, x_0,\te_R) = \f{u_{+,\te_R}' (z,x_0)}{u_{+,\te_R} (z,x_0)}, 
\quad x_0 \in (0,R),   \lb{7.9} \\
&m_{-,0} (z, x_0,\te_0) = \f{u_{-,\te_0}'(z,x_0)}{u_{-,\te_0} (z,x_0)}, 
\quad x_0 \in (0,R),    \lb{7.10}
\end{align}
and more generally, 
\begin{align}
m_{+,\alpha}(z,x_0,\te_R) &=\frac{-\sin(\alpha) + \cos(\alpha) m_{+,0}(z,x_0,\te_R)}
{\cos(\alpha) +\sin(\alpha) m_{+,0}(z,x_0,\te_R)}, \quad \alpha \in[0,\pi),   
\lb{7.11} \\
m_{-,\alpha}(z,x_0,\te_0) &=\frac{-\sin(\alpha) + \cos(\alpha) m_{-,0}(z,x_0,\te_0)}
{\cos(\alpha) +\sin(\alpha) m_{-,0}(z,x_0,\te_0)}, \quad \alpha \in[0,\pi).  \lb{7.12}
\end{align}
We note that $m_{+,\alpha}(\cdot,x_0,\te_R)$ and 
$- m_{-,\alpha}(\cdot,x_0,\te_0)$ are 
known to be Herglotz functions  (cf., e.g., \cite[Sect.\ 9.5]{CL85}), 
\cite[Sect.\ II.8]{LS75}, \cite[Sect.\ 6.5]{Pe88}, \cite[Ch.\ III]{Ti62}). 

Associated with \eqref{7.9}--\eqref{7.12} one then defines the $2 \times 2$ 
Weyl--Titchmarsh matrices (the fundamental ingredient for 
inverse spectral theory for operators of the type $\Hte$) by, 
\begin{align}
&M_{0}(z,x_0,\te_0,\te_R) = \big[M_{0, j,k}(z,x_0,\te_0,\te_R)\big]_{j,k=1,2}\no  \\
&= \big[m_{-,0}(z,x_0,\te_0)-m_{+,}(z,x_0,\te_R)\big]^{-1} \lb{7.13} \\
& \quad  \times\begin{bmatrix}
1 & \f12[m_{-,0}(z,x_0,\te_0)
+m_{+,0}(z,x_0,\te_R)] \\ 
\f12[m_{-,0}(z,x_0,\te_0)+m_{+,0}(z,x_0,\te_R)] 
& m_{-,0}(z,x_0,\te_0)m_{+,0}(z,x_0,\te_R) \end{bmatrix},    \no 
\end{align}
and more generally, by
\begin{align}
&M_{\alpha}(z,x_0,\te_0,\te_R) 
= \big[M_{\alpha, j,k}(z,x_0,\te_0,\te_R)\big]_{j,k=1,2}\no  \\
&= \big[m_{-,\alpha}(z,x_0,\te_0)-m_{+,\alpha}(z,x_0,\te_R)\big]^{-1} \lb{7.14} \\
& \quad  \times\begin{bmatrix}
1 & \f12[m_{-,\alpha}(z,x_0,\te_0)
+m_{+,\alpha}(z,x_0,\te_R)] \\ 
\f12[m_{-,\alpha}(z,x_0,\te_0)+m_{+,\alpha}(z,x_0,\te_R)] 
& m_{-,\alpha}(z,x_0,\te_0)m_{+,\alpha}(z,x_0,\te_R) \end{bmatrix}, \no \\
& \hspace*{11cm} \alpha \in[0,\pi).   \no 
\end{align}
By inspection,
\begin{align}
\det (M_{\alpha}(z,x_0,\te_0,\te_R))&= -1/4, \quad \alpha \in[0,\pi), \; 
z\in\bbC\setminus\bbR,   \lb{7.15} \\
\Im(M_{\alpha}(z,x_0,\te_0,\te_R)) &> 0, \quad \alpha \in[0,\pi), 
\; z\in\bbC_+,  \lb{7.16}
\end{align}
and hence $M_{\alpha} (\cdot,x_0,\te_0,\te_R)$ is a $2\times 2$ matrix-valued 
Herglotz function implying that $M_{\alpha, j,j}(\cdot,x_0,\te_0,\te_R)$ are 
Herglotz functions for $j=1,2$. In particular, the matrices  
$M_{\alpha} (z,x_0,\te_0,\te_R)$ can be shown to have Herglotz representations 
of the type \eqref{5.1}.  

For the connection of $M_{0} (z,x_0,\te_0,\te_R)$ and 
$M_{\alpha} (z,x_0,\te_0,\te_R)$ with the Green's function 
$G_{\te_0,\te_R} (z,x,x')$ of $\Hte$ we first introduce a bit of notation:  
\begin{align}
\partial_{1} G_{\te_0,\te_R}(z,x_0, x') & 
= \partial_{x_1} G_{\te_0,\te_R}(z,x_1,x')\big|_{x_1=x_0}, \no \\
\partial_{2} G_{\te_0,\te_R}(z,x,x_0) 
&= \partial_{x_2} G_{\te_0,\te_R} (z,x,x_2)\big|_{x_2 = x_0}, \lb{7.17} \\
\partial_{1} \partial_{2} G_{\te_0,\te_R} (z,x_0, x_0) & =
\partial_{x_1} \partial_{x_2} G_{\te_0,\te_R} (z,x_1,x_2)\big|_{x_1=x_0, x_2=x_0},\;
\mbox{ etc.} \no
\end{align}
The expressions \eqref{7.13} and \eqref{7.14} for $M_0 (z,x_0,\te_0,\te_R)$ and 
$M_\alpha(z,x_0,\te_0,\te_R)$ then can be rewritten as follows: 
\begin{align}
M_{0,1,1}(z,x_0,\te_0,\te_R) &= G_{\te_0,\te_R}(z,x_0,x_0). \lb{7.18} \\
M_{0,1,2}(z,x_0,\te_0,\te_R) &= M_{0,2,1}(z,x_0,\te_0,\te_R) \no \\
& = (1/2) (\partial_1 + \partial_2)
G_{\te_0,\te_R}(z,x_0\pm 0,x_0\mp 0),    \lb{7.19} \\
M_{0,2,2}(z,x_0,\te_0,\te_R) &= 
\partial_1 \partial_2 G_{\te_0,\te_R}(z,x_0,x_0), \lb{7.20} 
\end{align}
and 
\begin{align}
&M_{\alpha,1,1}(z,x_0,\te_0,\te_R)  \no \\
& \quad =\big(\cos(\alpha)
+\sin(\alpha)\partial_1\big)\big(\cos(\alpha)+
\sin(\alpha)\partial_2\big)G_{\te_0,\te_R}(z,x_0,x_0). \lb{7.21} \\
&M_{\alpha,1,2}(z,x_0,\te_0,\te_R)=M_{\alpha,2,1}(z,x_0,\te_0,\te_R) \no \\
& \quad =(1/2)\big((\cos(\alpha)+\sin(\alpha)\partial_1)
(-\sin(\alpha)+\cos(\alpha)\partial_2) \no \\ 
&\qquad +(-\sin(\alpha)+\cos(\alpha)\partial_1)
(\cos(\alpha)+\sin(\alpha)\partial_2)\big)G_{\te_0,\te_R}(z,x_0\pm 0,x_0\mp 0), 
\lb{7.22} \\
&M_{\alpha,2,2}(z,x_0,\te_0,\te_R)   \no \\
& \quad =\big(-\sin(\alpha)
+\cos(\alpha)\partial_1\big)\big(-\sin(\alpha)+
\cos(\alpha)\partial_2\big)G_{\te_0,\te_R}(z,x_0,x_0).  \lb{7.23} 
\end{align}

For relevant references in the context of \eqref{7.9}--\eqref{7.23}, we refer, 
for instance, to \cite[Ch.\ III]{CL90}, \cite[Ch.\ 9]{CL85}, 
\cite[App.\ J]{GH03}, \cite{GHSZ95}, \cite{GRT96}, \cite{GS96}, 
\cite[Ch.\ 2]{Le87}, \cite[Ch.\ 2]{LS75}, 
\cite[Ch.\ 6]{Pe88}, \cite[Chs.\ II, III]{Ti62}, and the references cited therein. 

A comparison of equations \eqref{7.1}, \eqref{7.2},  \eqref{7.4}, \eqref{7.5}, 
\eqref{7.7}, \eqref{7.8} with equations \eqref{7.18}--\eqref{7.20} and 
\eqref{7.21}--\eqref{7.23}, respectively, clearly exhibits the different character 
of $\Late (z)$ and $M_\alpha(z,x_0,\te_0,\te_R)$, $\alpha \in [0,\pi)$, despite 
the fact that both are $2 \times 2$ matrix-valued Herglotz functions whose 
associated matrix-valued measures contain all spectral information on $\Hte$. 
Additional differences are highlighted in Remark \ref{r3.3}, and we feel that 
$\Late (z)$ (and more generally, $\Lates (z)$) is deserving of a more detailed study.

\medskip

We conclude with a final observation:

%%%%%%%%%%%%%%%%%%%%%
\begin{remark} \lb{r6.1}
With only minor modifications, all results in this paper extend to general, regular three-coefficient differential expressions of the type
\begin{equation}
\f{1}{r} \bigg(- \f{d}{dx} p \f{d}{dx} + q\bigg), \quad x\in [0,R],
\end{equation}
where
\begin{equation}
p > 0, r > 0 \, \text{ a.e.\ on $(0,R)$}, \quad \f{1}{p}, q, \, r \in L^1((0,R); dx),
\end{equation}
generating Sturm--Liouville operator realizations in $L^2((0,R); r dx)$.
One just needs to consistently replace the derivative $f'$ for elements in operator domains and Wronskians by the first quasi-derivative $(pf')$.
\end{remark}
%%%%%%%%%%%%%%%%%%%%%

\medskip

%%%%%%%%%%%%%%%%%%%%%
\noindent {\bf Acknowledgments.}
We are indebted to Sergei Avdonin, Pavel Kurasov, and Mark Malamud for helpful discussions on this topic, and to Dorina Mitrea for a critical reading of this manuscript.
%%%%%%%%%%%%%%%%%%%%%

%%%%%%%%%%%%%%%%%%%%%%%%%%%%%%%%%%%%
%%%%%%%%%%%%%%%%%%%%%%%%%%%%%%%%%%%%

\end{document}